\documentclass[10pt]{amsart}
\usepackage{amsmath,amsfonts,amsthm}
\usepackage{color}
\usepackage{verbatim}
\usepackage{eepic,epic}
\usepackage{epsfig,subfigure,epstopdf}
\usepackage[]{graphics}
\usepackage[colorlinks,linktocpage,linkcolor=blue]{hyperref}

\newtheorem{theorem}{Theorem}[section]

\newtheorem{lemma}{Lemma}[section]

\newtheorem{assumption}{Assumption}[section]
\newtheorem{remark}{Remark}[section]

\numberwithin{equation}{section}

\def\II{{(D)}}

\newcount\icount

\def\DD#1#2{\icount=#1
  \ifnum\icount<1
  \,_{ 0}\kern -.1em D^{#2}_{\kern -.1em x}
  \else
  \,_{x}\kern -.2em D^{#2}_1
  \fi
}

\def\DDRI#1#2{\icount=#1
  \ifnum\icount<1
  \,_{-\infty}^{\kern 1em R}\kern -.2em D^{#2}_{\kern -.1em x}
  \else
  \,_{x}^R \kern -.2em D^{#2}_\infty
  \fi
}

\def\DDR#1#2{\icount=#1
  \ifnum\icount<1
 _{0}^{ \kern -.1em R} \kern -.2em D^{#2}_{\kern -.1em x}
  \else
 _{x}^{ \kern -.1em R} \kern -.2em D^{#2}_{\kern -.1em 1}
  \fi
}

\def\DDCI#1#2{\icount=#1
  \ifnum\icount<1
  \,_{-\infty}^{\kern 1em C}  \kern -.2em D^{#2}_{\kern -.1em x}
  \else
  \,_{x}^C \kern -.2em  D^{#2}_\infty
  \fi
}

\def\DDC#1#2{\icount=#1
  \ifnum\icount<1
  \,_{0}^C \kern -.2em  D^{#2}_{\kern -.1em x}
  \else
  \,_{x}^C \kern -.2em D^{#2}_1
  \fi
}

\def\Hd#1{\widetilde H^{#1}(D)}

\def\Hdi#1#2{\icount=#1
  \ifnum\icount<1
  \widetilde H_{L}^{#2}\II
  \else
  \widetilde H_{R}^{#2}\II
  \fi
}

\begin{document}

\title{A Petrov-Galerkin Finite Element Method for Fractional Convection-Diffusion Equations}
\author {Bangti Jin\and Raytcho Lazarov\and Zhi Zhou}
\address {Department of Computer Science, University College London, London, WC1E 2BT, UK
({bangti.jin@gmail.com,b.jin@ucl.ac.uk})}
\address {Department of Mathematics, Texas A\&M University, College Station, TX 77843-3368
(lazarov@math.tamu.edu)}
\address {Department of Applied Physics and Applied Mathematics,
Columbia University, 500 W. 120th Street, New York, NY 10027, USA (\texttt{zhizhou0125@gmail.com})}

\date{}
\maketitle

\begin{abstract}
In this work, we develop variational formulations of Petrov-Galerkin type for one-dimensional fractional boundary value
problems involving either a Riemann-Liouville or Caputo derivative of order $\alpha\in(3/2,
2)$ in the leading term and both convection and potential terms. They arise in the mathematical modeling of
asymmetric super-diffusion processes in heterogeneous media. The well-posedness of the formulations
and sharp regularity pickup of the variational solutions are established. A novel finite element method
is developed, which employs continuous piecewise linear finite elements and ``shifted'' fractional powers
for the trial and test space, respectively. The new approach has a number of distinct
features: It allows deriving optimal error estimates in both $L^2(D)$ and $H^1(D)$ norms; and
on a uniform mesh, the stiffness matrix of the leading term is diagonal and the resulting linear system is well conditioned.
Further, in the Riemann-Liouville case, an enriched FEM is proposed to improve the convergence.
Extensive numerical results are presented to
verify the theoretical analysis and robustness of the numerical scheme.
\end{abstract}

\section{Introduction}\label{sec:intro}
In this work, we consider the following one-dimensional fractional boundary value problem (FBVP)
\begin{equation}\label{eqn:fde}
  \begin{aligned}
    -\DD 0 \alpha u + bu' + qu &= f\quad \mbox{in } D = (0,1),\\
     u(0)=u(1) &= 0,
  \end{aligned}
\end{equation}
where the source term $f$ belongs to $L^2(D)$ or a suitable subspace, and $\DD 0 \alpha u $ denotes either the
left-sided Riemann-Liouville or Caputo fractional derivative of order $\alpha\in(3/2,2)$ defined in \eqref{eqn:RiemannCaputo}
below. The choice $\alpha\in(3/2,2)$ is to ensure the well-posedness of problem \eqref{eqn:fde} in the space $L^2(D)$,
and that the solution $u$ lies in $H_0^1(D)$ (see Section \ref{sec:var} for details) so that the $H^1(D)$
error estimate makes sense. Throughout, unless otherwise stated, we assume a convection
coefficient $b\in W^{1,\infty}(D)$ and a potential coefficient $q\in L^\infty(D)$. For $\alpha=2$, problem
\eqref{eqn:fde} recovers the canonical steady-state convection diffusion equation.

The interest in the model \eqref{eqn:fde} is motivated by anomalous diffusion in heterogeneous media.
Often it is used to describe super-diffusion processes, in which the mean squared variance grows at a rate faster than that in
a Gaussian process for normal diffusion. Microscopically, the fractional derivative describes long-range
interactions among particles and large particle jumps, and the choice
of the one-sided derivative reflects the asymmetry of the transport process
\cite{BensonWheatcraftMeerschaert:2000,delCastillo:2003}. The term $bu'$ describes convection under external flow field,
with a velocity $b$. The model has found successful applications in a number of areas, e.g., magnetized
plasma and subsurface flow.

\subsection{Review on existing studies}

The robust simulation of the model \eqref{eqn:fde} is challenging due to the nonlocality of the
fractional derivative and limited solution regularity. In the time dependent case, the finite difference method
(FDM) is predominant \cite{LynchCarreras:2003,TadjeranMeerschaertScheffler:2006,Sousa:2009,BaeumerKovacs:2015}; see
also \cite{JinLazarovPasciakZhou:2014} for a finite element method (FEM). Often, the stability of the schemes
and their error estimates were derived by assuming a sufficiently smooth solution.
 In this work, we focus on the steady-state model
\eqref{eqn:fde}, and review below the Riemann-Liouville and Caputo derivatives separately.

In the Riemann-Liouville case, Ervin and Roop \cite{ErvinRoop:2006} (see also \cite{FixRoop:2004}) gave a first
variational formulation of \eqref{eqn:fde} on the space $H_0^{\alpha/2}(D)$. The coercivity of the formulation was shown under suitable
conditions on the coefficients $b$ and $q$. However, in the presence of the convection term, for $\alpha\leq 3/2$, the variational
solution generally does not solve the equation in the $L^2(\Omega)$ sense, due to insufficient solution regularity.
A Galerkin FEM was also proposed, and error estimates were provided
by assuming that the solution is smooth, which remains completely open in the general case, and that the adjoint problem
has full regularity pickup, which generally does not hold.

In the absence of the convection term in \eqref{eqn:fde}, it was revisited in \cite{JinLazarovPasciak:2013a}, where sharp
regularity pickup was established for the first time and $H^{\alpha/2}(D)$ and $L^2(D)$ error
estimates, directly expressed in terms of the problem data, were provided for the Galerkin
FEM. However, the $L^2(D)$ error estimates are suboptimal. Wang and Yang \cite{WangYang:2013}
developed a stable Petrov-Galerkin formulation on the space $H^{\alpha-1}_0(D)\times
H_0^1(D)$, with a variable coefficient inside the fractional derivative. It was numerically
realized in \cite{WangYangZhu:2014}, where an $L^2(D)$-error estimate was provided. The problem
in \cite{WangYang:2013,WangYangZhu:2014} does not involve lower order terms, and its extension
to problem \eqref{eqn:fde} seems nontrivial.
In \cite{ZayernouriKarniadakis:2014}, Petrov-Galerkin formulations
for initial value problems for fractional ODEs and PDEs with a Riemann-Liouville derivative in time were studied.
Chen et al \cite{ChenShenWang:2014} proposed a spectral method for FBVPs of general order
without any lower order term, which merits exponentially convergence in
the $L^2(D)$ norm for suitably smooth data. However, the $L^2(D)$ error
estimate remains suboptimal \cite[Remark 5.2]{ChenShenWang:2014}.

One distinct feature of FBVPs with a Riemann-Liouville derivative is that the solution is usually weakly
singular, irrespective of the smoothness of the source term $f$. Thus the standard
 FEM converges slowly. There are several ways to improve the convergence, e.g., singularity
reconstruction \cite{JinZhou:2014} and transformation approach \cite{JinLazarovLuZhou:2015}.

The Caputo case is more delicate, and was scarcely studied. For example, for $\alpha\in(1,3/2]$,
the existence of a solution to problem \eqref{eqn:fde} with $f\in L^2(D)$ is unknown. This
is reminiscent of fractional diffusion with a Caputo derivative of order $\alpha\in (0,1/2)$ in
time \cite{GorenfloLuchkoYamamoto:2014}. The only variational formulation for the Caputo case
was derived in \cite{JinLazarovPasciak:2013a}. The trial space is $H^{\alpha/2}_0(D)$,
but the test space involves a nonlocal constraint. The stability and sharp
regularity pickup were shown, and a Galerkin FEM was proposed,
with optimal $H^{\alpha/2}(D)$ but suboptimal $L^2(D)$ error estimates.
Recently, Stynes and Gracia \cite{StynesGracia:2014} developed a FDM for \eqref{eqn:fde}
with a Caputo derivative and a Robin boundary condition, and derived an $L^\infty(D)$ rate.
See also \cite{ItoJinTakeuchi:2015} for a Legendre tau method, where
a suboptimal $L^2(D)$ error estimate was given for a smooth solution.

\subsection{Our contributions and the organization of the paper}\label{sec:contrib}
In this work, we shall develop proper variational formulations of Petrov-Galerkin type for the model
\eqref{eqn:fde}, and establish their well-posedness and sharp regularity pickup. For the choice
$\alpha\in (3/2,2)$, the variational solution satisfies \eqref{eqn:fde} in the $L^2(D)$ sense. Further, we
develop a novel FEM with continuous piecewise linear finite elements and ``shifted'' fractional
powers for the test and trial space, respectively.

The new FEM has a number of distinct features. First, the choice of the FEM test space allows us to derive
optimal error estimates in both $L^2(D)$ and $H^1(D)$ norms that are directly expressed in terms
of the problem data. In particular, this fills an outstanding gap in the theoretical analysis of FEMs for FBVPs,
for which only suboptimal $L^2(D)$ estimates were known.  Second, on a uniform mesh, the stiffness matrix of
the leading term is diagonal, and the resulting linear system is well conditioned.  To the best of our
knowledge, it is the first FEM with such desirable property.

In the Riemann-Liouville case, we also develop an enriched FEM based on a singularity reconstruction
technique \cite{CaiKim:2001,JinZhou:2014} to improve the convergence, by resolving the solution singularity
directly. We also derive optimal error estimates in both $H^1(D)$ and $L^2(D)$ norms, thereby improving
the results in \cite{JinZhou:2014}.

The rest of the paper is organized as follows. In Section \ref{sec:prelim} we recall preliminaries
on fractional calculus. The variational formulations are developed in Section \ref{sec:var}, where
the well-posedness and sharp regularity pickup are also studied. The new FEM and its implementation
details are described in Section \ref{sec:fem}, and optimal convergence rates are provided. Then we
present an enriched FEM for the Riemann-Liouville derivative in Section \ref{sec:sing}. Finally, in
Section \ref{sec:numer}, the theoretical analysis is numerically verified by extensive experiments, including
nonsmooth data. Throughout the notation $c$, with or without a subscript, denotes a generic constant, which may change at
different occurrences, but it is always independent of the mesh size $h$ and the solution $u$.

%%%%%%%%%%%%%%%%%%%%%%%%%%%%%%%%%%%%%%%%%%%%%%%%%%%%%%%%%%%%%%
\section{Preliminaries on fractional calculus}\label{sec:prelim}
%%%%%%%%%%%%%%%%%%%%%%%%%%%%%%%%%%%%%%%%%%%%%%%%%%%%%%%%%%%%%
We first briefly recall some preliminary facts on fractional calculus.
For any $\gamma > 0$ and $ f \in L^2(D)$ we define the left-sided Riemann-Liouville
fractional integral $_0\hspace{-0.3mm}I^{\gamma}_x f$ of order $\gamma$ by
\begin{equation}\label{eqn:RL-int}
 ({\,_0\hspace{-0.3mm}I^\gamma_x} f) (x)= \frac 1{\Gamma(\gamma)} \int_0^x (x-t)^{\gamma-1} f(t)dt,
\end{equation}
where $\Gamma(\cdot)$ is the Gamma function defined by $\Gamma(x)=\int_0^\infty t^{x-1}e^{-t}dt$ for $x>0$.
Then, for any $\beta>0$ with $n-1 < \beta < n$, $n\in\mathbb{N}$, the
left-sided Riemann-Liouville and Caputo derivatives of order $\beta$ of $f\in H^n(D)$, denoted by $\DDR 0
\beta f$ and $\DDC0\beta f$, are respectively defined by
\cite{KilbasSrivastavaTrujillo:2006,Podlubny_book}
\begin{equation}\label{eqn:RiemannCaputo}
\DDR0\beta u =\frac {d^n} {d x^n} \big({_0\hspace{-0.3mm}I^{n-\beta}_x} u\big)\quad \mbox{and} \quad
  \DDC 0 \beta u = {_0\hspace{-0.3mm}I^{n-\beta}_x}\big(\frac {d^nu} {dx^n }\big).
\end{equation}
Analogously we define the right-sided Riemann-Liouville integral $_xI_1^\gamma f$ by
\begin{equation*}
  ({_x\hspace{-0.3mm}I^\gamma_1} f) (x)= \frac 1{\Gamma(\gamma)}\int_x^1 (t-x)^{\gamma-1}f(t)\,dt
\end{equation*}
and the right-sided derivatives of order $\beta$ by
\begin{equation*}
  \DDR1\beta u =(-1)^n\frac {d^n} {d x^n} \big({_x\hspace{-0.3mm}I^{n-\beta}_1} u\big)
  \quad \mbox{and} \quad \DDC 1 \beta u = (-1)^n {_x\hspace{-0.3mm}I^{n-\beta}_1}\big(\frac {d^nu} {dx^n }\big).
\end{equation*}
The following formula for change of integration order is valid
\cite[pp. 76, Lemma 2.7]{KilbasSrivastavaTrujillo:2006}
\begin{equation}\label{eqn:change-order}
  ({\,_0\hspace{-0.3mm}I^\gamma_x} \psi,\varphi)=(\psi,{\,_x\hspace{-0.3mm}I^\gamma_1} \varphi)
  \quad\forall\psi,\varphi\in L^2(D),
\end{equation}
where $(\cdot,\cdot)$ denotes the $L^2(D)$ inner product.

Now we introduce some function spaces. For any $\beta\ge 0$, we denote $H^\beta(D)$ to be
the Sobolev space of order $\beta$ on the unit interval $D$ \cite{AdamsFournier:2003}, and $\Hd \beta $ the set of
functions in $H^\beta(D)$ whose extension by zero to $\mathbb{R}$ is in $H^\beta(\mathbb{R})$. Likewise, we define
$\Hdi 0 \beta$ (respectively, $\Hdi 1 \beta$) to be the set of functions $u$ whose extension by zero, denoted by
$\tilde{u}$, is in $H^\beta(-\infty,1)$ (respectively, $H^\beta(0,\infty)$). Further, for $u\in \Hdi 0
\beta$, we set $\|u\|_{\Hdi 0\beta}:=\|\tilde{u}\|_{H^\beta(-\infty,1)}$, and similarly the norm
in $\Hdi 1 \beta$.

The next theorem collects some useful properties of fractional integral and differential operators
(see \cite[pp. 73, Lemma 2.3]{KilbasSrivastavaTrujillo:2006} \cite[Theorems 2.1 and 3.1]{JinLazarovPasciak:2013a}).
\begin{theorem}\label{thm:fracop}
The following statements hold.
\begin{itemize}
  \item[$\mathrm{(a)}$] The integral operators $_0I_x^\beta$ and $_xI_1^\beta$ satisfy the semigroup property.
  \item[$\mathrm{(b)}$] The operators $\DDR0\beta$ and $\DDR1\beta $ extend continuously to bounded operators
     from $\Hdi 0 \beta$ and $\Hdi 1\beta$, respectively, to $L^2(D)$.
  \item[$\mathrm{(c)}$] For any $s,\beta\geq 0$, the operator
$_0I_x^\beta$ is bounded from $\Hdi0 s$ to $\Hdi0{\beta+s}$,
and $_xI_1^\beta$ is bounded from $\Hdi1 s$ to $\Hdi1{\beta+s}$.
\end{itemize}
\end{theorem}

We shall also need the following two results. The first asserts the equivalence of the two fractional
derivatives on suitable function spaces, and the second gives an algebraic property of fractional-order Sobolev spaces.
\begin{lemma}[\cite{JinLazarovPasciak:2013a}, Lemma 4.1]
\label{l:caprl}
For $u\in \Hdi 0 1$  and $\beta\in (0,1)$,
$\DDR 0\beta u = {_0I_x^{1-\beta}}(u^\prime)$.
Similarly, for $u\in \Hdi 1 1$  and $\beta\in (0,1)$,
$\DDR1\beta u =-{_x I_1^{1-\beta}}(u^\prime)$.
\end{lemma}

\begin{lemma}[\cite{JinLazarovPasciak:2013a}, Lemma 4.6]\label{lem:Hsalg}
Let $0<s\leq1, s\neq 1/2$. Then for any $u\in\Hd s\cap L^\infty(D)$
and $v\in H^s(D)\cap L^\infty(D)$, the product $uv$ is in $\Hd s$.
\end{lemma}

%%%%%%%%%%%%%%%%%%%%%%%%%%%%%%%%%%%%
\section{Variational formulation and regularity}\label{sec:var}
%%%%%%%%%%%%%%%%%%%%%%%%%%%%%%%%%%%%%%%%%
Now we develop proper variational formulations for
problem \eqref{eqn:fde}, and establish their stability and sharp regularity pickup.
We shall discuss the Riemann-Liouville and Caputo cases separately.
\subsection{Variational formulation in the Riemann-Liouville case}\label{subsec:var-RL}
First we consider the case $b,q\equiv0$. Then it was
shown in \cite[Section 3]{JinLazarovPasciak:2013a} that for $f \in L^2(D)$,
the solution $u$ of \eqref{eqn:fde} is given by
\begin{equation}\label{eqn:solrep-RL}
 u = -{_0\hspace{-0.3mm}I^{\alpha}_x} f
     + ({_0\hspace{-0.3mm}I^{\alpha}_x}f) (1)x^{\alpha-1} \in \Hdi0 {\alpha-1+\beta},
\end{equation}
for any $\beta\in[2-\alpha,1/2)$. Thus for $\alpha\in(3/2,2)$, $u\in \Hd 1$.
Further, for $\varphi\in C_0^\infty(D)$, by
\eqref{eqn:change-order} and Lemma \ref{l:caprl}, there holds
\begin{equation*}
\begin{split}
 (\DDR0 \alpha u, \varphi)&= (({_0\hspace{-0.3mm}I_x^{2-\alpha}}u)'',\varphi)=-(({_0\hspace{-0.3mm}I_x^{2-\alpha}}u)',\varphi')\\
&=-({_0\hspace{-0.3mm}I_x^{2-\alpha}}u',\varphi')=-(u',{_x\hspace{-0.3mm}I_1^{2-\alpha}}\varphi')=(u',~\DDR1{\alpha-1} \varphi).
\end{split}
\end{equation*}
This motivates us to define a bilinear form $a(\cdot,\cdot):\Hd1\times\Hd{\alpha-1}\rightarrow\mathbb{R}$ by
\begin{equation}\label{eqn:a}
  a(u,\varphi):=-(u',~\DDR1{\alpha-1} \varphi).
\end{equation}
Throughout, we set $U=\Hd 1$ and $V=\Hd {\alpha-1}$ below, and denote by $U^*$
etc. the dual space of $U$ etc., and the norms on $U$ etc. by $\|\cdot\|_U$ etc. Further,
we also denote the duality pairing by $\langle\cdot,\cdot\rangle$.

Now we state our first result on the stability of the variational formulation.
\begin{lemma}\label{lem:inf-sup-a}
The bilinear form $a(\cdot,\cdot)$ in \eqref{eqn:a} satisfies the inf-sup condition:
\begin{equation}\label{eqn:inf-sup1}
 \sup_{\varphi\in V} \frac{a(u,\varphi)}{\| \varphi \|_V} \ge c_0 \| u \|_U.
\end{equation}
\end{lemma}
\begin{proof}
For any fixed $u\in U$, let $\varphi_u ={\DDR1 {2-\alpha}u}-(\DDR1 {2-\alpha}u)(0)(1-x)^{\alpha-1}$.
Clearly, $\varphi_u(0)=\varphi_u(1)=0$. For $\beta\in[2-\alpha,1/2)$, the term $(1-x)^{\alpha-1} \in
\Hdi 1{\alpha-1+\beta}$. Meanwhile, by Lemma \ref{l:caprl} and Theorem \ref{thm:fracop}(c), for $u\in U$, there holds
\begin{equation*}
  \|  \DDR1 {2-\alpha} u \|_{H^{\alpha-1}(D)} = \| {_x\hspace{-0.3mm}I_1^{\alpha-1}}u' \|_{H^{\alpha-1}(D)} \le c \| u\|_U,
\end{equation*}
and thus $\varphi_u \in V$ and is a valid test function. Further,
\begin{equation*}
\begin{split}
\| \varphi_u  \|_{V} & \le  c \left(\|  \DDR1 {2-\alpha} u \|_{H^{\alpha-1}(D)}+ c|\DDR1 {2-\alpha}u(0)| \|(1-x)^{\alpha-1}\|_{H^{\alpha-1}(D)}\right)\\
&\le  c\left( \| u \|_U + | ({_x\hspace{-0.3mm}I_1^{\alpha-1}}u')(0) | \right) \le c \|  u\|_U.
\end{split}
\end{equation*}
Since $u(0)=u(1)=0$, $(u',~\DDR1{\alpha-1}(1-x)^{\alpha-1})=c_\alpha(u',1)=0$, and
we derive the following inf-sup condition:
\begin{equation*}
  \begin{aligned}
    \sup_{\varphi\in V} & \frac{a(u,\varphi)}{\| \varphi \|_V}
        \ge \frac{-(u',\,{\DDR 1 {\alpha-1}} \varphi_u)}{\|\varphi_u\|_V}
    %& \ge c_0\frac{-(u',~\DDR1{\alpha-1}(-{_x\hspace{-0.3mm}I_1^{\alpha-1}}u'-(\DDR1 {2-\alpha}u)(0)(1-x)^{\alpha-1}))}{\| u \|_U}\\
    \geq c_0\frac{-(u',~\DDR1{\alpha-1}(-{_x\hspace{-0.3mm}I_1^{\alpha-1}}u'))}{\| u\|_U} = c_0 \| u\|_U.
  \end{aligned}
\end{equation*}
\end{proof}

For any nonzero $\varphi \in V$, we choose $u_\varphi = {_x\hspace{-0.3mm}I_1^{2-\alpha}}\varphi -
({_x\hspace{-0.3mm}I_1^{2-\alpha}}\varphi) (0)(1-x)$, which is nonzero and belongs to $U$. Then
\begin{equation*}
   a(u_\varphi,\varphi)=-(u_\varphi',~\DDR1 {\alpha-1}\varphi) = \| u_\varphi \|_U^2 >0.
\end{equation*}
It implies that if $a(u,\varphi)=0$ for all $u\in U$, then $\varphi=0$. This and Lemma
\ref{lem:inf-sup-a} give the stability of the variational problem for the case $b,
q \equiv 0$. Namely, given any $F\in V^*$, there exists a unique solution $u\in U$ such that
\begin{equation*}
  a(u,\varphi) = \langle F,\varphi\rangle \quad \forall \varphi\in V.
\end{equation*}

We now turn to the general case $b, q\not \equiv 0$. The corresponding variational formulation
reads: given any $F\in V^*$, find $u\in U$ such that
\begin{equation}\label{eqn:var-RL}
  A(u,\varphi) = \langle F,\varphi\rangle\quad \forall\varphi\in V,
\end{equation}
where the bilinear form $A(\cdot,\cdot):U\times V\to \mathbb{R}$ is defined by
\begin{equation*}
  A(u,\varphi)=  a(u,\varphi)+ (bu',\varphi) + (qu,\varphi).
\end{equation*}
To study the bilinear form $A(\cdot,\cdot)$, we make the following uniqueness assumption.
\begin{assumption} \label{ass:riem}
Let the bilinear form $A(\cdot,\cdot):U\times V\to\mathbb{R}$ satisfy
\begin{itemize}
 \item[{$\mathrm{(a)}$}]  The problem of finding $u \in U$ such that $A(u,\varphi)=0$ for all $\varphi \in V$
           has only the trivial solution $u\equiv 0$.
 \item[{$(\mathrm{a}^\ast)$}] The problem of finding $\varphi \in V$ such that $A(u,\varphi)=0$ for all $u \in U$
    has only the trivial solution $\varphi \equiv 0$.
\end{itemize}
\end{assumption}

\begin{theorem}\label{thm:wpreg}
Let $b, q\in L^\infty(D)$ and Assumption \ref{ass:riem} hold. Then for any $F\in V^*$,
there exists a unique solution $u\in U$ to problem \eqref{eqn:var-RL}.
\end{theorem}
\begin{proof}
In case of $b,q\equiv0$, the assertion follows from Lemma \ref{lem:inf-sup-a}.
In general, the proof is based on Petree-Tartar's lemma \cite[pp. 469, Lemma A.38]{ern-guermond}.
To this end, we define two operators $S \in \mathcal{L}(U;V^*)$ and $T \in \mathcal{L}(U;V^*)$ by
\begin{equation*}
  \langle Su,\varphi\rangle=A(u,\varphi)\quad \text{and}\quad(Tu,\varphi)=-(bu',\varphi)-(qu,\varphi),
\end{equation*}
respectively. By Assumption \ref{ass:riem}(a), the operator $S$ is injective. The compactness of the
operator $T$ follows from $b,q\in L^\infty(D)$ and the compact embedding from
$L^2(D)$ into $V^\ast$. Further, by Lemma \ref{lem:inf-sup-a}, we deduce that for any $u\in U$
\begin{equation*}
  \begin{aligned}
   \|u\|_U &\le c\sup_{\varphi\in V} \frac{ a(u,\varphi)} {\|\varphi\|_V}
   \le c\sup_{\varphi\in V} \frac{ A(u,\varphi)} {\|\varphi\|_V}+c\sup_{\varphi\in V}\frac{ -(bu',\varphi)-(qu,\varphi)} {\|\varphi\|_{V}}\\
   & =c(\|Su \|_{V^*}+\| Tu \|_{V^*}).
  \end{aligned}
\end{equation*}
Then by Petree-Tartar's lemma, the image of the operator $S$ is closed; equivalently,
there exists a constant $c_0>0$ such that
\begin{equation}\label{inf-sup-RL}
  c_0  \|u\|_U \le \sup_{\varphi\in V } \frac {A(u,\varphi)} {\|\varphi\|_V}\quad \forall u\in U.
\end{equation}
This and Assumption \ref{ass:riem}($a^*$) show that the operator $S:U\to V^*$ is bijective, i.e.,
there exists a unique solution $u\in U$ to problem \eqref{eqn:var-RL}.
\end{proof}

\begin{theorem}\label{thm:reg-RL}
Let $b,q\in L^\infty(D)$ and $\langle F,\varphi\rangle=(f,\varphi)$ for some $f\in L^2(D)$, and
Assumption \ref{ass:riem} hold. Then there exists a unique solution $u \in \Hdi0
{\alpha-1+\beta} \cap \Hd1$ to problem \eqref{eqn:var-RL} for any $\beta\in[2-\alpha,1/2)$ and it satisfies
\begin{equation}\label{eqn:reg-RL}
  \| u \|_{\Hdi0 {\alpha-1+\beta}} \le c \|f\|_{L^2(D)}.
\end{equation}
\end{theorem}
\begin{proof}
By Theorem \ref{thm:wpreg}, we have $\|  u \|_{\Hd 1} \le c\|f\|_{L^2(D)}.$
Then we rewrite \eqref{eqn:var-RL} as $ -\DDR0 \alpha u
= \widetilde f$ with $\widetilde f = f - bu' - qu$. Since $b,q\in
L^\infty(D)$ and $u\in \Hd1$, $ bu' +qu \in L^2(D)$, and $\widetilde
f\in L^2(D)$. Then \eqref{eqn:reg-RL} follows from \eqref{eqn:solrep-RL}
and Theorem \ref{thm:fracop}(c).
\end{proof}

Next we consider the adjoint problem: for a given $F \in U^\ast$, find $w\in V$ such that
\begin{equation}\label{eqn:var-RL-adj}
 A(\varphi,w)=\langle \varphi,F \rangle \quad \forall \varphi \in U.
\end{equation}
The inf-sup condition for problem \eqref{eqn:var-RL-adj} with $b,q\equiv0$ was shown in \cite[Theorem 5.5]{WangYang:2013}.
The stability for $b\in W^{1,\infty}(D)$ and $q\in L^\infty(D)$ follows from Assumption
\ref{ass:riem} and the argument in the proof of Theorem \ref{thm:wpreg}.
If $\langle  \varphi,F \rangle=(\varphi,f)$ for some $f\in L^2(D)$, the solution $w$
for $b,q\equiv0$ is given by
$w=-{_x\hspace{-0.3mm}I_1^{\alpha}}f + ({_x\hspace{-0.3mm}I_1^{\alpha}}f) (0)(1-x)^{\alpha-1}.$
This implies  $w\in \Hdi1{\alpha-1+\beta} $ with $\beta \in [2-\alpha,1/2)$.

To rigorously analyze the adjoint problem, we first extend the domain of
the operator ${_x\hspace{-0.3mm}I_1^{\gamma}}$ to the space $\Hd{-\gamma}$, $\gamma\in(0,1/2)$,
the dual space of $\Hd{\gamma}\equiv H^\gamma(D)$, by means
of duality (see \cite{McBride:1979} for an in depth treatment). Specifically, we define ${_x\hspace{-0.3mm}I_1^{\gamma}}$ on $\Hd{-\gamma}$ by
$({_x\hspace{-0.3mm}I_1^{\gamma}} \varphi, \psi ) := \langle \varphi, {_0\hspace{-0.3mm}I_x^{\gamma}}\psi
\rangle$ for all $\varphi \in \Hd{-\gamma},\psi\in L^2(D)$, and for $\alpha>\gamma$, ${_x\hspace{-0.3mm}I_1^{\alpha}} \varphi := {_x\hspace{-0.3mm}I_1^{\alpha-\gamma}}{_x\hspace{-0.3mm}I_1^{\gamma}}\varphi$
 for all $\varphi\in \Hd{-\gamma}$.
Next we verify the consistency relation for $\alpha\in(3/2,2)$
\begin{equation*}
  \DDR1{\alpha}{_x\hspace{-0.3mm}I_1^{\alpha}}\varphi = \varphi \in \Hd{-\gamma} ~~\text{with} ~~\gamma\in(0,1/2).
\end{equation*}
In fact, for any $\psi\in C_0^\infty(D)$, by Theorem \ref{thm:fracop}(a) and \eqref{eqn:change-order}, there holds
\begin{equation*}
\begin{split}
  \langle \DDR1{\alpha}{_x\hspace{-0.3mm}I_1^{\alpha}}\varphi, \psi\rangle
  & = \langle ({_x\hspace{-0.3mm}I_1^{2-\alpha}}{_x\hspace{-0.3mm}I_1^{\alpha}}\varphi)'', \psi\rangle
   \stackrel{\text{def}}{=} \langle ({_x\hspace{-0.3mm}I_1^{2-\gamma}}{_x\hspace{-0.3mm}I_1^{\gamma}}\varphi)'', \psi\rangle
  =  \langle {_x\hspace{-0.3mm}I_1^{2-\gamma}}{_x\hspace{-0.3mm}I_1^{\gamma}}\varphi, \psi''\rangle\\
  & %= ( {_x\hspace{-0.3mm}I_1^{2-\gamma}}{_x\hspace{-0.3mm}I_1^{\gamma}}\varphi, \psi'')
  =  ( {_x\hspace{-0.3mm}I_1^{\gamma}}\varphi, {_0\hspace{-0.3mm}I_x^{2-\gamma}} \psi'')
  \stackrel{\text{def}}{=} \langle \varphi, {_0\hspace{-0.3mm}I_x^{2}} \psi''\rangle = \langle \varphi, \psi \rangle.
\end{split}
\end{equation*}
Now we show that, for $\alpha>\gamma>0$, the operator ${_x\hspace{-0.3mm}I_1^{\alpha}}$ is bounded from $\Hd{-\gamma}$
to $\Hdi1{\alpha-\gamma}$. Indeed, by Theorem \ref{thm:fracop}(a) and (c), we have
\begin{equation}\label{int-weak-bound}
\begin{split}
  \|{_x\hspace{-0.3mm}I_1^{\alpha}}\varphi\|_{\Hdi1{\alpha-\gamma}}
  :&= \|{_x\hspace{-0.3mm}I_1^{\alpha-\gamma}}{_x\hspace{-0.3mm}I_1^{\gamma}}\varphi\|_{\Hdi1{\alpha-\gamma}}
  \le c \|{_x\hspace{-0.3mm}I_1^{\gamma}}\varphi\|_{L^2(D)}\\
   & = c \sup_{\psi\in L^2(D)}\frac{\langle  \varphi, {_0\hspace{-0.3mm}I_x^{\gamma}}\psi  \rangle}{\| \psi \|_{L^2(D)}}
     \le c \| \varphi  \|_{\Hd{-\gamma}}.
\end{split}
\end{equation}
Thus the representation $w=-{_x\hspace{-0.3mm}I_1^{\alpha}}f+({_x\hspace{-0.3mm}I_1^{\alpha}}f)(0)(1-x)^{\alpha-1}$ is
a solution to \eqref{eqn:var-RL-adj} with $b,q\equiv0$ and $F = f \in \Hd{-\gamma}$, $\gamma\in(0,1/2)$.
In sum, we have the following lemma. %We summarize the discussions in a lemma.
\begin{lemma}\label{lem:adj-RL}
Let $F=f\in \Hd{-\gamma}$, $\gamma\in(0,1/2)$, and $b,q\equiv0$. Then
$w=-{_x\hspace{-0.3mm}I_1^{\alpha}}f+({_x\hspace{-0.3mm}I_1^{\alpha}}f)(0)(1-x)^{\alpha-1}$
is a solution of problem \eqref{eqn:var-RL-adj} and for $\beta\in[2-\alpha,1/2)$
$$ \| w  \|_{\Hdi1{\alpha-1+\beta}} \le c\| f \|_{\Hd{-\gamma}}.  $$
\end{lemma}

Now we can state a regularity result for the adjoint problem \eqref{eqn:var-RL-adj}.
\begin{theorem}\label{thm:reg-RL-adj}
Let $b\in W^{1,\infty}(D)$, $q\in L^\infty(D)$ and $\langle F,\varphi\rangle=(f,\varphi)$ for
some $f\in L^2(D)$ and Assumption \ref{ass:riem} hold. Then there exists a unique
solution $w \in \Hdi1{\alpha-1+\beta} \cap \Hd{\alpha-1}$ to problem \eqref{eqn:var-RL-adj}
for any $\beta\in[2-\alpha,1/2)$ and it satisfies
\begin{equation*}
  \| w \|_{\Hdi1 {\alpha-1+\beta}} \le c \|f\|_{L^2(D)}.
\end{equation*}
\end{theorem}
\begin{proof}
By the inf-sup condition, there exists a solution $w\in \Hd {\alpha-1}$ to \eqref{eqn:var-RL-adj}.
Next we rewrite it into $a(\varphi,w) = \langle \varphi, \widetilde{f}\rangle,$ with $\widetilde{f}=
f+(bw)'-qw$, with $\|\widetilde f\|_{\Hd{\alpha-2}}\leq c\|f\|_{L^2(D)}$. By Lemma
\ref{lem:adj-RL}, $w=-{_x\hspace{-0.3mm}I_1^{\alpha}}\widetilde{f}+({_x\hspace{-0.3mm}I_1^{\alpha}}
\widetilde{f})(0)(1-x)^{\alpha-1}$. Since $ \alpha>3/2$, by Theorem \ref{thm:fracop}(c),
${_x\hspace{-0.3mm}I_1^{\alpha}}\widetilde{f}\in\Hdi 1{2\alpha-2} \subset \Hdi 1 {\alpha-1+\beta}$,
for any $\beta\in[2-\alpha,1/2)$. Hence, the desired estimate holds:
$ \| w\|_{\Hdi1 {\alpha-1+\beta}} \le c\|  \widetilde f \|_{\Hd{\alpha-2}} \leq c \| f \|_{L^2(D)}.$
\end{proof}

%%%%%%%%%%%%%%%%%%%%%%%%%%%%%%%%%%%%%%%%%%%%%%%%%%%%%%%%%%%%%%%%
\subsection{Variational formulation in the Caputo case}
%%%%%%%%%%%%%%%%%%%%%%%%%%%%%%%%%%%%%%%%%%%%%%%%%%%%%%%%%%%%%%%%
Now we consider the Caputo case. For $b,q\equiv0$,
the solution of \eqref{eqn:fde} is given by \cite[Section 3]{JinLazarovPasciak:2013a}
\begin{equation}\label{eqn:solrep-Cap}
 u=- {_x\hspace{-0.3mm}I_1^{\alpha}}f+ ({_x\hspace{-0.3mm}I_1^{\alpha}}f)(0)x \in H^\alpha(D)\cap \Hd1.
\end{equation}
Recall the defining relation $\DDC0\alpha u = {\DDR 0 \alpha} (u-u'(0)x-u(0))$
\cite[pp. 91]{KilbasSrivastavaTrujillo:2006}. Hence, with $u(0)=0$,  for any
$\varphi \in \widetilde H_R^\infty(D)$, upon integration by parts, we have
\begin{equation*}
   ({\DDC 0 \alpha}u,\varphi) = ({\DDR 0 \alpha}(u-u'(0)x),\varphi) = (u',{\DDR1 {\alpha-1}}\varphi) - u'(0)({\DDR 0 \alpha}x,\varphi).
\end{equation*}
Since  $u'(0)$ does not make sense on the space $\Hd 1$, we choose the test
function $\varphi$ such that $({\DDR 0 \alpha} x, \varphi)=0$, i.e., $(x^{1-\alpha},\varphi)=0$. Hence,
we let $ W=\{\varphi\in \Hdi1{\alpha-1}: (\varphi,x^{1-\alpha})=0\}$, and
introduce a bilinear form $a(\cdot,\cdot):U\times W\to \mathbb{R}$ by
\begin{equation}\label{eqn:a-C}
a(u,\varphi) = -(u',~\DDR1 {\alpha-1} \varphi).
\end{equation}
The only difference from the Riemann-Liouville case lies in the test
space $W$: in the Caputo case, it  involves an integral constraint $(\varphi,x^{1-\alpha})=0$.

The following lemma gives the stability of the bilinear form.
\begin{lemma}\label{lem:inf-sup-a-C}
The bilinear form $a(\cdot,\cdot)$ in \eqref{eqn:a-C} satisfies
\begin{equation}\label{eqn:inf-sup1-C}
 \sup_{\varphi \in W} \frac{a(u,\varphi)}{\| \varphi \|_{W}} \ge c_0 \| u \|_U.
\end{equation}
\end{lemma}
\begin{proof}
For a fixed $u \in \Hd1$, let
$\varphi_u = - {_x\hspace{-0.3mm}I_1^{\alpha-1}}u' + c_1 (1-x)^{\alpha-1},$
where $c_1$ is chosen such that $(\varphi_u,x^{1-\alpha})=0$, i.e.,
$c_1=-(\DDR1 {2-\alpha}u,x^{1-\alpha})/( (1-x)^{\alpha-1}, x^{1-\alpha}).$
Since $u'\in L^2(D)$ and $(1-x)^{\alpha-1}\in \Hdi 1{\alpha-1+\beta}$
with $\beta\in[2-\alpha,1/2)$, by Theorem \ref{thm:fracop}(c),
we have $\varphi_u\in W$. Further, by Theorem \ref{thm:fracop},
\begin{equation*}
  \| \varphi_u \|_{W} \le c(\| {_x\hspace{-0.3mm}I_1^{\alpha-1}}u' \|_{\Hdi1{\alpha-1}}
  + \|{_x\hspace{-0.3mm}I_1^{\alpha-1}}u' \|_{L^\infty(D)})
    \le  c\| u \|_U.
\end{equation*}
The desired assertion follows from the argument in the proof of Lemma \ref{lem:inf-sup-a}.
\end{proof}

For any $\varphi\neq 0 \in W$, let $u_\varphi = {_x\hspace{-0.3mm}I_1^{2-\alpha}}\varphi$. Since $\varphi\in W$, $
({_x\hspace{-0.3mm}I_1^{2-\alpha}}\varphi)(0)=0$, and obviously $u_\varphi(1)=0$, we deduce $ u_\varphi\in U $. Further,
\begin{equation}\label{eqn:cap-inj}
  a(u_\varphi,\varphi)=-(u_\varphi',~\DDR1 {\alpha-1}\varphi) = \| \DDR1 {\alpha-1}\varphi \|_{L^2(D)}^2 = \|\varphi\|_W^2 >0.
\end{equation}
Hence if $a(u,\varphi)=0$ for all $u\in U$, then $\varphi=0$.  This and Lemma \ref{lem:inf-sup-a-C}
imply the variational stability. In the case $b,q\ne0$, the variational formulation reads:
given any $F\in W^*$, find $u\in U$ such that
\begin{equation}\label{eqn:var-Caputo}
  A(u,\varphi)= \langle F, \varphi \rangle \quad \forall\varphi \in W,
\end{equation}
where  the bilinear form $A(\cdot,\cdot):U\times W\to \mathbb{R}$ is given by
\begin{equation*}
  A(u,\varphi) = a(u,\varphi) +(bu',\varphi) + (qu,\varphi).
\end{equation*}
To analyze problem \eqref{eqn:var-Caputo}, we assume the unique solvability.
\begin{assumption}\label{ass:caputo}
Let the bilinear form $A(\cdot,\cdot):U\times W\to\mathbb{R}$ satisfy
\begin{itemize}
 \item[{$\mathrm{(a)}$}]The problem of finding $u \in U$ such that $A(u,\varphi)=0$ for all $\varphi \in W$
           has only the trivial solution $u\equiv 0$.
 \item[{$(\mathrm{a}^\ast)$}] The problem of finding $\varphi \in W$ such that $A(u,\varphi)=0$ for all $u \in U$
    has only the trivial solution $\varphi\equiv 0$.
\end{itemize}
\end{assumption}

Under Assumption \ref{ass:caputo}, we have the following existence result. The proof is identical
to that of Theorem \ref{thm:wpreg} and hence omitted.
\begin{theorem} \label{thm:wpreg-C}
Let Assumption \ref{ass:caputo} hold and $b,q\in L^\infty(D)$. Then for any $F\in W^*$, there
exists a unique solution $u\in U$ to \eqref{eqn:var-Caputo}.
\end{theorem}

\begin{theorem} \label{thm:regcap}
Let Assumption \ref{ass:caputo} hold and $s\in[0,1/2)$. If $\langle F,v\rangle=(f,v)$
for some $f\in {H^s}(D)$, and $b,q\in L^\infty(D)\cap H^s(D)$, then
the solution $u\in U$ of \eqref{eqn:var-Caputo} is in $H^{\alpha+s}(D)\cap \Hd1$
and further it satisfies
\begin{equation*}
  \|u\|_{H^{\alpha+s}(D)}\leq c\|f\|_{{H^s}(D)}.
\end{equation*}
\end{theorem}
\begin{proof}
We show the regularity, by rewriting \eqref{eqn:var-Caputo} as
$-\DDC0{\alpha} u = \widetilde{f}$, with $\widetilde f=f-bu'-qu.$
By Lemma \ref{lem:Hsalg}, $bu'\in L^2(D)$ and $qu\in H^s(D) $, and thus $\widetilde f\in L^2(D)$.
By \eqref{eqn:solrep-Cap}, $u\in H^\alpha(D) \cap \Hd 1 $, and by Lemma \ref{lem:Hsalg},
$bu'$ and $qu\in H^s(D)$, and thus $\widetilde f\in \Hd s$.
The assertion follows, by appealing again to Theorem \ref{thm:fracop} and \eqref{eqn:solrep-Cap}.
\end{proof}

\begin{remark}
In the Riemann-Liouville case, the solution is limited to $H^{\alpha-1+\beta}(D)$, $\beta\in [2-\alpha,1/2)$,
irrespective of the smoothness of $f$, whereas in the Caputo case, it can be
made smoother, by imposing suitable smoothness on $q$, $b$ and  $f$.
\end{remark}

Now we consider the adjoint problem: given any $F\in U^*$,
find $w\in W$ such that
\begin{equation}\label{eqn:var-adj-C}
 A(\varphi, w) = \langle \varphi , F\rangle \quad   \forall \varphi \in U .
\end{equation}

The next result gives the well-posedness of the adjoint problem.
\begin{theorem}\label{thm:wp-adj-C}
Let Assumption \ref{ass:caputo} hold, $b \in W^{1,\infty}(D)$ and $q\in L^\infty(D)$. Then for any $F\in U^*$, there
exists a unique solution $w\in W$ to \eqref{eqn:var-adj-C}.
\end{theorem}
\begin{proof}
First consider the case $b, q\equiv0$. For any $w\in W$, ${_x\hspace{-0.3mm}I_1^{2-\alpha}}
w \in U$, and thus Theorem \ref{thm:fracop}(a) yields
\begin{equation*}
 w = -({_x\hspace{-0.3mm}I_1^{1}}w)' = -({_x\hspace{-0.3mm}I_1^{\alpha-1}}{_x\hspace{-0.3mm}I_1^{2-\alpha}}w)'
  = -{_x\hspace{-0.3mm}I_1^{\alpha-1}}({_x\hspace{-0.3mm}I_1^{2-\alpha}}w)' = {_x\hspace{-0.3mm}I_1^{\alpha-1}}{\DDR1{\alpha-1}}w.
\end{equation*}
Hence by Theorem \ref{thm:fracop}(c), for $w\in W$, there holds
\begin{equation}\label{norm-equiv}
 \| w \|_W = \| {_x\hspace{-0.3mm}I_1^{\alpha-1}}
 {\DDR1{\alpha-1}}w \|_W\le c \| \DDR1{\alpha-1}w \|_{L^2(D)}.
\end{equation}
Next, with $u_w={_x\hspace{-0.3mm}I_1^{2-\alpha}}w \in U$ and by Theorem \ref{thm:fracop}(c),
$\| u_w \|_U = \|  \DDR1{\alpha-1} w  \|_{L^2(D)}\le c\| w \|_W.$ Then the inf-sup condition
in case of $b,q\equiv0$ follows from \eqref{eqn:cap-inj} and \eqref{norm-equiv}
\begin{equation}\label{eqn:inf-sup-C2}
  \sup_{u\in U} \frac{a(u,w)}{\| u \|_U} \ge \frac{a(u_w,w)}{\| u_w \|_U} \ge c_0\| w \|_W.
\end{equation}
Moreover, if $u\in U$ and $a(u,w)=0$ for all $w\in W$, then $u=0$ by \eqref{eqn:inf-sup1-C}.
This and \eqref{eqn:inf-sup-C2} indicate that problem \eqref{eqn:var-adj-C} has a unique
solution $w\in W$. The well-posedness for $b,q\neq 0$ follows by repeating the
argument for Theorem \ref{thm:wpreg}.
\end{proof}

Next we derive the regularity pickup for problem \eqref{eqn:var-adj-C}. If $\langle \varphi,
F\rangle = (\varphi,f)$ for some $f\in L^2(D)$, the strong form of the problem reads
\begin{equation}\label{eqn:strong-adj-C}
  -\DDR1 {\alpha} w - (bw)' +qw = f,
\end{equation}
with $w(1) = 0$ and $(w,x^{1-\alpha}) = 0$. For $b, q\equiv 0$, the solution $w$ is given by
\begin{equation*}
  w= -  {_x\hspace{-0.3mm}I_1^{\alpha}}f + c_f (1-x)^{\alpha-1} \quad
  \text{with} ~~ c_f =  ({_x\hspace{-0.3mm}I_1^{\alpha}}f,x^{1-\alpha})/(  (1-x)^{\alpha-1}, x^{1-\alpha}). %B(\alphapha,2-\alphapha).
\end{equation*}
The constant $c_f$ can be bounded by
$|c_f| \le c | ( {_x\hspace{-0.3mm}I_1^{\alpha}}f,x^{1-\alpha})| \le c \| f \|_{L^2(D)}.$
Hence $w\in \Hdi1{\alpha-1+\beta} $ with $\beta \in [2-\alpha,1/2)$. The general case can
be analyzed analogously to Theorem \ref{thm:reg-RL-adj}, and then
we have the following regularity result.

\begin{theorem}\label{thm:reg-adj-C}
Let Assumption \ref{ass:caputo} hold, and $b\in W^{1,\infty}(D)$, $q\in L^\infty(D)$.
Then with $\langle \varphi,F\rangle=(\varphi,f)$ for some $f\in L^2(D)$, the solution $w$ to
problem \eqref{eqn:var-adj-C} is in $\Hd{\alpha-1}\cap\Hdi1{\alpha-1+\beta}$ for any
$\beta\in[2-\alpha,1/2)$ and further it satisfies
\begin{equation*}
  \|  w \|_{\Hdi1{\alpha-1+\beta}} \le c \| f \|_{L^2(D)}.
\end{equation*}
\end{theorem}

\begin{remark}
The adjoint problem for both derivatives is of
Riemann-Liouville type, but with different constraints: the Riemann-Liouville case
involves $w(0)=0$, whereas the Caputo case
$(w,x^{1-\alpha})=0$. Hence, their regularity pickup is identical.
\end{remark}

%%%%%%%%%%5%%%%%SSSSSSSSSSSSS
\section{Finite element approximation}\label{sec:fem}
%%%%%%%%%%%%%%SSSSSSSSSSSSS

Now we apply the variational formulations
to the numerical approximation of problem \eqref{eqn:fde}. We shall develop novel FEMs
using continuous piecewise linear finite elements and ``shifted'' fractional powers
$(x_{i}-x)_+^{\alpha-1}$ for the trial and test space, respectively, analyze their stability,
and derive optimal error estimates for the approximation in the $L^2(D)$ and $H^1(D)$ norms.

\subsection{Finite element spaces and their approximation properties}
First we introduce the finite element spaces based on a uniform partition of the domain $D$,
with nodes $0=x_0<x_1<\ldots<x_m=1$ and a mesh size $h=1/m$. We then define $U_h$ to be the
set of functions in $U$ which are linear when restricted to the subintervals $[x_i,x_{i+1}]$,
$i=0,\ldots,m-1$, i.e.,
\begin{equation*}
   U_h = \left\{\psi_h \in U:~~ \psi_h=ax+b,~~~x\in [x_i,x_{i+1}]\right\} .
\end{equation*}

To define the test spaces, we first introduce ``shifted'' fractional powers
\begin{equation*}
 \varphi_i(x) = \left \{
    \begin{tabular}{cc}
     $(x_{i}-x)^{\alpha-1}$ & $ x \le x_{i} $ \\
     $ 0  $                & $ x > x_{i} $
    \end{tabular} \right \}
    :=(x_{i}-x)^{\alpha-1}\chi_{[0,x_{i}]}(x),\quad i= 1,2,\ldots,m,
\end{equation*}
where $\chi_S$ denotes the characteristic function of a set $S$. The basis functions
$\varphi_i(x)$ can be written as $\varphi_i(x)=\Gamma(\alpha){_x\hspace{-0.3mm}I_1^{\alpha-1}}
\chi_{[0,x_{i}]}(x)$, i.e., the fractional derivative $\DDR 1 {\alpha-1} \varphi_i$
is piecewise constant. Clearly, $\varphi_i \in \Hdi1{\alpha -1+\beta}$ for any $\beta\in
[2-\alpha,1/2)$. Then we define the following finite-dimensional subspaces $V_h\subset
V$ and $W_h\subset W$
\begin{equation*}
V_h=\text{span}\{\varphi_i\}_{i=1}^{m}\cap V \quad \text{and} \quad
W_h=\text{span}\{\varphi_i\}_{i=1}^{m}\cap W,
\end{equation*}
as the test space for the Riemann-Liouville and Caputo derivative, respectively.

Next we introduce two operators $P_V:L^2(D)\to L^2(D)$ and $P_W:L^2(D)\to L^2(D)$,
respectively, by: for any $\psi\in L^2(D)$,
\begin{equation*}
  P_V \psi = \psi - \Gamma(\alpha)({_x\hspace{-0.3mm}I_1^{\alpha-1}}\psi)(0) \quad\mbox{and}\quad
  P_W \psi = \psi - ({_x\hspace{-0.3mm}I_1^{\alpha-1}}\psi,x^{1-\alpha})/\Gamma(2-\alpha).
\end{equation*}
\begin{lemma}\label{lem:PvPw}
The operators $P_V$ and $P_W$ are bounded. Further, for $\psi\in \Hd 1$, $P_V{\DDR 1 {\alpha-1} \psi}={\DDR 1 {\alpha-1} \psi}$,
and for $\psi\in W\cap \Hdi 1 1$, $P_W{\DDR 1 {\alpha-1} \psi}={\DDR 1 {\alpha-1} \psi}$.
\end{lemma}
\begin{proof}
Let $\delta=\alpha-1$. For any $\psi\in L^2(D)$, the boundedness of $P_V$ follows from
\begin{equation*}
  \| P_V \psi\|_{L^2(D)} \le \|\psi\|_{L^2(D)} + c|({_x\hspace{-0.3mm}I_1^{\delta}}\psi)(0)|
  \le c\| \psi \|_{L^2(D)},
\end{equation*}
The boundedness of $P_W$ is similar.
For $\psi\in \Hd 1$, with $\varphi_\psi={\DDR 1 \delta}\psi$, by Lemma \ref{l:caprl}
and Theorem \ref{thm:fracop}(a), ${_x\hspace{-0.3mm}I_1^{\delta}}\varphi_\psi =-
{_x\hspace{-0.3mm}I_1^{\delta}}({_x\hspace{-0.3mm}I_1^{1-\delta}}\psi)'
= -({_x\hspace{-0.3mm}I_1^{1}}\psi)'=\psi,$
and thus $({_x\hspace{-0.3mm}I_1^\delta}\varphi_\psi)(0)=0$, which yields
$P_V \varphi_\psi = \varphi_\psi - \Gamma(\delta+1)({_x\hspace{-0.3mm}I_1^{\delta}}\varphi_\psi)(0) = \varphi_\psi$.
Likewise, for $\psi\in W\cap \Hdi 1 1$,
$({_x\hspace{-0.3mm}I_1^{\delta}} \varphi_\psi,x^{-\delta}) = (\psi,x^{-\delta})=0$, which gives directly
$P_W \varphi_\psi = \varphi_\psi$.
\end{proof}

Now we can state important approximation properties of these spaces.
\begin{lemma}\label{fem-interp-U}
Let the mesh be quasi-uniform and $1\leq \gamma \leq 2$, and $\delta=\alpha-1\in (1/2,1)$.
If $u \in H^\gamma(D) \cap \Hd{1}$, then
\begin{equation*}%\label{approx-Uh}
\inf_{\psi_h \in U_h} \| u - \psi_h \|_U \le ch^{\gamma - 1} \|u\|_{H^\gamma(D)}.
\end{equation*}
Further, if $u \in \Hdi1{\gamma}\cap V$, then
\begin{equation*}
\inf_{\psi_h \in V_h} %\| \DDR1 {\delta} (u - \psi_h)\|_{L^2(D)}
\| u - \psi_h\|_V
\le ch^{\min(1,\gamma-\delta)} \|u\|_{H^\gamma(D)}.
\end{equation*}
Similarly, if $u \in \Hdi1{\gamma}\cap W$, then
\begin{equation*}
\inf_{\psi_h \in W_h} \|u-\psi_h\|_W %\| \DDR1 {\delta} (u - \psi_h)\|_{L^2(D)}
\le ch^{\min(1,\gamma -\delta)} \|u\|_{H^\gamma(D)}.
\end{equation*}
\end{lemma}
\begin{proof}
Let $\Pi_h u \in U_h$ be the standard Lagrange interpolant of $u\in \Hd 1$ so that
\begin{equation*}
 \inf_{\psi_h \in U_h} \| u -\psi_h \|_U \le \| u - \Pi_h u \|_U.
\end{equation*}
The first estimate follows from the approximation
property of $\Pi_h$ \cite[Corollary 1.109, pp. 61]{ern-guermond}.
Next we consider the space $V_h$. For $u\in \Hdi1 {\gamma}$,
by Theorem \ref{thm:fracop}(c), we have
\begin{equation}\label{eqn:fyu}
    \| ({_x\hspace{-0.3mm}I_1^{1-\delta}}u)' \|_{\Hdi1{\gamma-\delta} } \le c\| {_x\hspace{-0.3mm}I_1^{1-\delta}}u \|_{\Hdi1{\gamma+1-\delta} }
    \le c\| u \|_{\Hdi1 {\gamma}}.
\end{equation}
Hence $ \varphi_u:=-({_x\hspace{-0.3mm}I_1^{1-\delta}}u)'= {\DDR1{\delta}}u $ belongs to $\Hdi1 {\gamma-\delta}$.
On the space $L^2(D)$, we define a projection operator $\Pi_0:L^2(D)\to L^2(D)$ by (with $h_i=x_{i+1}-x_i$)
\begin{equation}\label{l2-project}
  \Pi_0 \psi (x) = \frac{1}{h_i}\int_{x_i}^{x_{i+1}}  \psi(s) \,ds \quad x\in (x_i, x_{i+1}].
\end{equation}
By the definitions of $\Pi_{0}$ and $P_V$,  ${_x\hspace{-0.3mm}I_1^{\delta}}
(P_V\Pi_{0}\psi) \in V_h$ for any $\psi \in L^2(D)$. Then using \eqref{eqn:fyu}, the
property of $\Pi_0$ and Lemma \ref{lem:PvPw}, we deduce
\begin{equation*}
    \begin{split}
    \inf_{\psi_h\in V_h}\|u-\psi_h\|_V &\leq  \|  \DDR1{\delta}(u  -{_x\hspace{-0.3mm}I_1^{\delta}}(P_V\Pi_{0}\varphi_u))\|_{L^2(D)}
    =  \|  \varphi_u - P_V \Pi_{0}\varphi_u\|_{L^2(D)}\\
    &=  \| P_V( \varphi_u - \Pi_{0}\varphi_u)\|_{L^2(D)}
    \le  c \|\varphi_u - \Pi_{0}\varphi_u\|_{L^2(D)}  \\
    &   \le    c h^{\min(1,\gamma-\delta)}\| \varphi_u \|_{H^{\gamma-\delta}(D)}
    \le c h^{\min(1,\gamma-\delta)}\| u \|_{\Hdi1 {\gamma}}.
     \end{split}
\end{equation*}
Last, by the definitions of $\Pi_{0}$ and $P_W$, ${_x\hspace{-0.3mm}I_1^{\delta}}(P_W\Pi_{0}\psi)
\in W_h$ for any $\psi \in L^2(D)$. Now the $L^2(D)$ stability of $P_W$ and the identity $P_W\varphi_u
=\varphi_u$ from Lemma \ref{lem:PvPw} yield
\begin{equation*}
    \begin{split}
    \inf_{\psi_h\in W_h}\|u-\psi_h\|_W & \leq
    \|  \DDR1{\delta}(u  -{_x\hspace{-0.3mm}I_1^{\delta}}(P_W\Pi_{0}\varphi_u))\|_{L^2(D)}
    =  \| P_W( \varphi_u - \Pi_{0}\varphi_u)\|_{L^2(D)}\\
  &\le c\|\varphi_u - \Pi_{0}\varphi_u\|_{L^2(D)}\le ch^{\min(1,\gamma-\delta)}\| u \|_{\Hdi1{\gamma}}.
  \end{split}
\end{equation*}
\end{proof}

%%%%%%%%%%%%%%%%%%%%%%%%%%%ssssssss
\subsection{Error estimates in the Riemann-Liouville case}
%%%%%%%% sssssssssss
The finite element problem is: given any $F\in V^*$, find $u_h\in U_h$ such that
\begin{equation}\label{gal:rl}
A(u_h,\varphi_h) = \langle F,\varphi_h\rangle\quad \forall \varphi_h\in V_h.
\end{equation}
We shall establish optimal error estimates for the approximation
$u_h$ for $\langle F,v\rangle\equiv (f,v) $ with $f\in L^2(D)$, using several technical lemmas.

A first lemma shows the stability of problem \eqref{gal:rl} when $b,q\equiv0$.

\begin{lemma}\label{lem:inf-sup-disc-rl-triv}
For the bilinear form $a(\cdot,\cdot)$ in \eqref{eqn:a}, there holds
\begin{equation}\label{discr-inf-sup-triv}
   \sup_{\varphi_h \in V_h} \frac{a(\psi_h,\varphi_h)}{ \|\varphi_h\|_{V}} \ge c\|\psi_h\|_U \quad \forall \psi_h\in U_h,
\end{equation}
and the finite element problem: Find $u_h\in U_h$ such that
\begin{equation}\label{eqn:discp-triv-fem}
a(u_h,\varphi_h)=(f,\varphi_h) \quad \forall \varphi_h\in V_h,
\end{equation}
has a unique solution.
\end{lemma}
\begin{proof}
For any $\psi_h\in U_h$, let $\varphi_h=-{_x\hspace{-0.3mm}I_1^{\alpha-1}}(P_V \psi_h')$.
By Theorem \ref{thm:fracop}, $\varphi_h \in V_h$ and
\begin{equation*}
\begin{split}
  a(\psi_h,\varphi_h)&=-\left(\psi_h', ~\DDR1{\alpha-1}
  \left(-{_x\hspace{-0.3mm}I_1^{\alpha-1}}(P_V \psi_h')\right)\right)
  = (\psi_h',P_V \psi_h')\\
   &= (\psi_h', \psi_h') -  \Gamma(\alpha)({_x\hspace{-0.3mm}I_1^{\alpha-1}}\psi_h')(0) (\psi_h',1)
   =\| \psi_h\|_U^2.
\end{split}
\end{equation*}
Further, by Lemma \ref{lem:PvPw} and $\varphi_h(0)=0$, there holds
\begin{equation*}
 \| \varphi_h  \|_V = \| {_x\hspace{-0.3mm}I_1^{\alpha-1}}(P_V \psi_h') \|_V
  \le c\| P_V \psi_h' \|_{L^2(D)} \le c\| \psi_h \|_U.
\end{equation*}
Thus the condition \eqref{discr-inf-sup-triv} holds. Since the stiffness matrix is
square, the existence follows from uniqueness, which is a direct consequence of \eqref{discr-inf-sup-triv}, and
thus problem \eqref{eqn:discp-triv-fem} has a unique solution $u_h\in U_h$.
\end{proof}

Next, we introduce the (adjoint) Ritz projection $R_h: V\to V_h$ defined by
\begin{equation*}%\label{eqn:Ritz}
a(\psi_h,R_h\varphi)=a(\psi_h,\varphi) \quad \forall \psi_h \in U_h.
\end{equation*}

\begin{lemma}\label{lem:Ritz}
The operator $R_h$ is well-defined and satisfies for any $\beta\in(2-\alpha,1/2)$
\begin{equation}\label{eqn:Ritz-prop}
  \begin{aligned}
    \|R_h \varphi\|_V & \le c\| \varphi\|_V,\\
    \| \varphi- R_h \varphi \|_{L^{2}(D)} & \le c h^{\alpha-2+\beta}   \| \varphi \|_V.
  \end{aligned}
\end{equation}
\end{lemma}
\begin{proof}
For any $\varphi_h \in V_h$, let $\psi_h= {_x\hspace{-0.3mm}I_1^{2-\alpha}} \varphi_h -({_x\hspace{
-0.3mm}I_1^{2-\alpha}} \varphi_h) (0)(1-x)$. Then $\psi_h(0)=\psi_h(1)=0$, and
$\psi_h'= -{\DDR1{\alpha-1}}\varphi_h+({_x\hspace{-0.3mm}I_1^{2-\alpha}\varphi_h)(0)} $,
i.e., $\psi_h$ is the primitive of a piecewise constant function, hence $\psi_h \in U_h$.
Since $ \DDR1 {2-\alpha}\psi_h= \varphi_h - ({_x\hspace{-0.3mm}I_1^{2-\alpha}} \varphi_h) (0)(\DDR1{2-\alpha}(1-x))$
and $ {\varphi_h(0)}=0$, we deduce
$\DDR1 {2-\alpha}\psi_h(0)= c_0 ({_x\hspace{-0.3mm}I_1^{2-\alpha}} \varphi_h)(0)$, with $c_0=-(\DDR1{2-\alpha}(1-x))(0)\neq0.$
Thus the following bound holds
\begin{equation*}
   |({_x\hspace{-0.3mm}I_1^{2-\alpha}} \varphi_h)(0)| \le c |\DDR1 {2-\alpha}\psi_h(0)| = c |({_x\hspace{-0.3mm}I_1^{\alpha-1}}\psi_h')(0) |\le c\| \psi_h\|_U,
\end{equation*}
which directly yields
\begin{equation*}
  \| \varphi_h \|_{V} =  \|\DDR1{\alpha-1} \varphi_h \|_{L^2(D)} \le c\left(\| \psi_h'  \|_{L^2(D)} +c |({_x\hspace{-0.3mm}I_1^{2-\alpha}} \varphi_h)(0)|\right) \le c\| \psi_h\|_U.
\end{equation*}
This and the identity ${a(\psi_h,\varphi_h)={ \|\psi_h\|_U^2}}$ give the following discrete inf-sup condition
\begin{equation*}%\label{discr-inf-sup2}
   \sup_{\psi_h \in U_h}
  \frac{a(\psi_h,\varphi_h)}{ \|\psi_h\|_{U}}\geq c\|\varphi_h\|_V \quad \forall\varphi_h\in V_h,
\end{equation*}
which shows that $R_h$ is well-defined. Then the $\Hd{\alpha-1}$-stability of $R_h$ follows:
\begin{equation*}
      \|R_h \varphi\|_V \leq c\sup_{\psi_h \in U_h} \frac{a(\psi_h,R_h \varphi)}{ \|\psi_h\|_{U}}
         =c\sup_{\psi_h \in U_h} \frac{a(\psi_h,\varphi)}{ \|\psi_h\|_{U}} \le c\| \varphi \|_V.
\end{equation*}
Next let $g$ be the solution to \eqref{eqn:var-RL} with $F=\varphi-R_h\varphi$.
By Theorem \ref{thm:reg-RL}, $\|g\|_{\Hdi 0 {\alpha-1+\beta}}\leq c\|\varphi-R_h\varphi\|_{L^2(D)}$,  $\beta\in(2-\alpha,
1/2)$. Then Galerkin
orthogonality, and Lemma \ref{fem-interp-U} give
\begin{equation*}
  \begin{aligned}
    \| \varphi- R_h \varphi \|_{L^{2}(D)}^2 & = a(g,\varphi- R_h \varphi)
     \le c \inf_{\psi_h\in U_h}\|  g- \psi_h \|_U \| \varphi-R_h\varphi \|_V\\
    &\le ch^{\alpha-2+\beta} \|\varphi-R_h\varphi\|_{L^2(D)} \| \varphi \|_V.
  \end{aligned}
\end{equation*}
\end{proof}

The next result gives the stability for the discrete variational formulation in the
general case, using a kickback technique \cite{Schatz-1974, JinLazarovPasciak:2013a}.
\begin{lemma}\label{lem:disinfsup:riem}
Let Assumption \ref{ass:riem} hold, $f\in L^2(D)$, and $b,q\in
L^\infty(D)$.   Then there exists an $h_0>0$ such that for all $h\le h_0$
\begin{equation}\label{discr-inf-sup}
  c \|\psi_h\|_U \le \sup_{\varphi_h \in V_h} \frac{A(\psi_h,\varphi_h)}{ \|\varphi_h\|_{V}} \quad \forall\psi_h\in U_h.
\end{equation}
For such $h$, the finite element problem: Find $u_h\in U_h$ such that
\begin{equation}\label{discp}
A(u_h,\varphi_h)=(f,\varphi_h) \quad \forall \varphi_h\in V_h,
\end{equation}
has a unique solution.
\end{lemma}
\begin{proof}
For any $\psi_h \in U_h$, by the inf-sup condition \eqref{inf-sup-RL}, there holds
\begin{equation*}
 c_0 \| \psi_h \|_{U}\le \sup_{\varphi \in V} \frac{A(\psi_h,\varphi)}{\| \varphi \|_{V}}
\leq \sup_{\varphi \in V} \frac{A(\psi_h,\varphi-R_h\varphi)}{\| \varphi \|_{V}}
+  \sup_{\varphi \in V} \frac{A(\psi_h,R_h\varphi)}{\| \varphi \|_{V}}=:\mathrm{I}+\mathrm{II}.
\end{equation*}
By Lemma \ref{lem:Ritz}, we have for $\beta\in(1-\alpha/2,1/2)$
\begin{equation*}
   \mathrm{I} = \sup_{\varphi \in V} \frac{(b\psi_h' + q\psi_h,\varphi-R_h\varphi)}{\| \varphi \|_{V}}
    \le  c\sup_{\varphi \in V} \frac{\|\psi_h \|_{U}  \| \varphi-R_h\varphi \|_{L^2(D)}}{\| \varphi \|_{V}}\le c_1h^{\alpha-2+\beta}\| \psi_h \|_{U}.
\end{equation*}
Meanwhile, the second term $\mathrm{II}$ can be bounded by \eqref{inf-sup-RL}
and Lemma \ref{lem:Ritz}:
\begin{equation*}
  \mathrm{II}\le c\sup_{\varphi \in V} \frac{A(\psi_h,R_h\varphi)}{\| R_h\varphi\|_{V}}
  =c\sup_{\varphi_h \in V_h} \frac{A(\psi_h,\varphi_h)}{\| \varphi_h \|_{V}}.
\end{equation*}
By choosing an $h_0$ such that $c_1 h_0^{\alpha-2+\beta}=c_0/2$ we get the discrete inf-sup
condition \eqref{discr-inf-sup}, and the unique existence of the solution to \eqref{discp} follows directly.
\end{proof}

Last we give some error estimates for the adjoint problem.
\begin{lemma}\label{lem:adjoint-error-RL}
Let Assumption \ref{ass:riem} hold, $f\in L^2(D)$, $b \in W^{1,\infty}(D)$ and $q\in L^\infty(D)$, and
$w$ be the solution of the adjoint problem \eqref{eqn:var-RL-adj}. Then there holds
\begin{equation}\label{approx-adj-RL}
    \inf_{\psi_h \in V_h} \|w-\psi_h\|_{L^2(D)} + \inf_{\psi_h \in V_h} \|w-\psi_h\|_V %\|\DDR1 {\alpha-1} (w-\psi_h)\|_{L^2(D)}
    \le c h \|f\|_{L^2(D)}.
\end{equation}
\end{lemma}
\begin{proof}
By the solution representation, $w= w^r  +  w^s$ and $w^s=\mu (1-x)^{\alpha-1}$, where the regular
part $w^r \in \Hdi1{\alpha}$, and $\mu\in \mathbb{R}$. By Theorem \ref{thm:reg-RL-adj}, there holds
\begin{equation}\label{eqn:wrlef}
\begin{split}
  \| w^r \|_{\Hdi1{\alpha}}&\le \| {_x\hspace{-0.3mm}I_1^{\alpha}}f \|_{\Hdi1{\alpha}} +  \| {_x\hspace{-0.3mm}I_1^{\alpha}}(bw)' \|_{\Hdi1{\alpha}} +
  \| {_x\hspace{-0.3mm}I_1^{\alpha}}(qw) \|_{\Hdi1{\alpha}}\\
  &\le c ( \| f \|_{L^2(D)} + \| bw \|_{\Hdi1{1}} + \| qw \|_{L^2(D)}) \le c\| f \|_{L^2(D)}.
\end{split}
\end{equation}
Let $\varphi_{w^r} = {\DDR1{\alpha-1} w^r}$,  $\varphi_{w^s} = {\DDR1{\alpha-1} w^s}$ and $\varphi_w=\varphi_{w^r}+\varphi_{w^s}$.
By Lemma \ref{lem:PvPw}, we have
\begin{equation*}
 \begin{split}
  \| \varphi_w - P_V \Pi_0 \varphi_w\|_{L^2(D)} &= \|  P_V(\varphi_w - \Pi_0\varphi_w)\|_{L^2(D)} \le c \| \varphi_w - \Pi_0 \varphi_w \|_{L^2(D)}  \\
  & \le c\| \varphi_{w^r} - \Pi_0 \varphi_{w^r}\|_{L^2(D)} + c\| \varphi_{w^s} - \Pi_0 \varphi_{w^s}\|_{L^2(D)}.
 \end{split}
\end{equation*}
In view of $\varphi_{w^r}\in \Hdi1{1}$, by Lemma \ref{fem-interp-U} and \eqref{eqn:wrlef}, we deduce
\begin{equation*}
 \| \varphi_{w^r} - \Pi_0 \varphi_{w^r}\|_{L^2(D)} \le c h \|  \varphi_{w^r} \|_{\Hdi1{1}}\le ch\| w^r \|_{\Hdi1{\alpha}}\le ch\| f \|_{L^2(D)}.
\end{equation*}
Meanwhile since $\varphi_{w^s}$ is a constant, $\| \varphi_{w^s} - \Pi_0 \varphi_{w^s}\|_{L^2(D)} = 0.$ Then by letting
$\psi_h:= {_x\hspace{-0.3mm}I_1^{\alpha-1}} P_V \Pi_0 \varphi_w \in V_h $ and Lemma \ref{lem:PvPw}, there holds
\begin{equation*}
 \inf_{\psi_h \in V_h}\|w-\psi_h\|_V\leq \|\DDR1 {\alpha-1} (w-\psi_h)\|_{L^2(D)}  \le \| \varphi_w - P_V \Pi_0 \varphi_w\|_{L^2(D)}  \le ch\| f \|_{L^2(D)}.
\end{equation*}
Since $w \in \Hdi11$, there holds $w=  {{_x\hspace{-0.3mm}I_1^{\alpha-1}}}{\DDR1{\alpha-1}}
w = {_x\hspace{-0.3mm}I_1^{\alpha-1}}\varphi_w$. Hence
\begin{equation*}
\begin{split}
  \inf_{\psi_h \in V_h} \|w-\psi_h\|_{L^2(D)}  &\le \| {_x\hspace{-0.3mm}I_1^{\alpha-1}} (\varphi_w - P_V \Pi_0 \varphi_w)\|_{L^2(D)}  \\
  & \le c \| \varphi_w - P_V \Pi_0 \varphi_w\|_{L^2(D)}   \le ch\| f \|_{L^2(D)}.
\end{split}
\end{equation*}
\end{proof}

\begin{remark}
The $L^2(D)$  estimate in Lemma \ref{lem:adjoint-error-RL} is suboptimal, but the
$H^{\alpha-1}(D)$ estimate suffices deriving an optimal $L^2(D)$ estimate in Theorem \ref{thm:femrl}.
\end{remark}

Now, we state the main theorem of this part, i.e., optimal error estimates in the $L^2(D)$ and
$H^1(D)$ norms for the Petrov-Galerkin FEM in the Riemann-Liouville case.
\begin{theorem}\label{thm:femrl}
Let Assumption \ref{ass:riem} hold, $f\in L^2(D)$, $b\in W^{1,\infty}(D)$, and $q\in L^\infty(D)$. Then
there exists an $h_0>0$ such that for all $h\le h_0$, the solution $u_h\in U_h$ to
problem \eqref{discp} satisfies for any $\beta\in(2-\alpha,1/2)$,
\begin{equation*}
   \|u-u_h\|_{L^2(D)} + h\|u-u_h\|_U  \le ch^{\alpha-1+\beta} \|f\|_{L^2(D)}.
\end{equation*}
\end{theorem}
\begin{proof}
The error estimate in the $\Hd1$-norm follows from Cea's lemma, \eqref{discr-inf-sup}
and Galerkin orthogonality. Specifically, for any $h\leq h_0$ and $\psi_h \in U_h$,
by \eqref{discr-inf-sup}
\begin{equation*}
\begin{split}
    \| u_h- \psi_h \|_{U} &\le c\sup_{\varphi_h \in V_h} \frac{A(u_h-\psi_h,\varphi_h)}{ \|\varphi_h\|_{V}}
    \le c\sup_{\varphi_h \in V_h} \frac{A(u-\psi_h,\varphi_h)}{ \|\varphi_h\|_{V}}
    \le c\| u-\psi_h \|_{U}.
\end{split}
\end{equation*}
Hence the triangle inequality yields for any $\psi_h \in U_h$
\begin{equation*}
    \| u-u_h \|_U \le \|  u-\psi_h \|_U + \|\psi_h - u_h\|_U
    \le c\| u-\psi_h \|_U.
\end{equation*}
Then the $\Hd1$-estimate follows from Theorem \ref{thm:reg-RL} and Lemma \ref{fem-interp-U}
\begin{equation*}
    \|u-u_h\|_U \le \inf_{\psi_h\in U_h}c\|u-\psi_h \|_U
    \le ch^{\alpha-2+\beta} \|f\|_{L^2(D)},
\end{equation*}
with $\beta \in (2-\alpha,1/2)$. Next let $w$ be the solution of
problem \eqref{eqn:var-RL-adj} with  $F=u-u_h$. By
Theorem \ref{thm:reg-RL} and Lemmas \ref{fem-interp-U} and \ref{lem:adjoint-error-RL}, we deduce
\begin{equation*}
\begin{split}
    \| u-u_h \|_{L^2(D)}^2 &= A( u-u_h, w)
    \le c\|u-u_h \|_U \inf_{\varphi_h \in V_h} \|w-\varphi_h\|_V\\%\| \DDR1 {\alpha-1}(w-\varphi_h) \|_{L^2(D)}\\
    & \le c h^{\alpha-1+\beta} \| f \|_{L^2(D)}\| u-u_h  \|_{L^2(D)}.
\end{split}
\end{equation*}
\end{proof}

\begin{remark}
Since the solution $u$ is in $\Hd1 \cap H^{\alpha-1+\beta}(D)$ with $\beta\in(2-\alpha,1/2)$,
both $L^2(D)$ and $\Hd 1$ error estimates are optimal. This is in stark contrast with that
in \cite{JinLazarovPasciak:2013a}, where the $L^2(D)$-error estimate suffers from one half order loss.
\end{remark}

\subsection{Error estimates in the Caputo case}
Here the finite element problem reads: given any $F\in W^*$, find $u_h\in U_h$ such that
\begin{equation}\label{eqn:fem-C}
    A(u_h, \varphi_h) = \langle F,\varphi_h \rangle \quad  \forall\varphi_h \in W_h.
\end{equation}

First we prove the stability of problem \eqref{eqn:fem-C} for the case $b,\,q\equiv0$.
\begin{lemma}\label{lem:inf-sup-disc-cap-triv}
Let $a(\cdot,\cdot)$ be the bilinear form in \eqref{eqn:a-C}. Then there holds
\begin{equation}\label{eqn:discr-inf-sup-triv-c}
   \sup_{\varphi_h \in W_h} \frac{a(\psi_h,\varphi_h)}{ \|\varphi_h\|_{W}} \ge c\|\psi_h\|_U \quad \forall \psi_h\in U_h,
\end{equation}
and the finite element problem: Find $u_h\in U_h$ such that
\begin{equation*}
a(u_h,\varphi_h)=(f,\varphi_h) \quad \forall \varphi_h\in W_h,
\end{equation*}
has a unique solution.
\end{lemma}
\begin{proof}
For any fixed $\psi_h\in U_h$, let $\varphi_h=-{_x\hspace{-0.3mm}I_1^{\alpha-1}}(P_W \psi_h')$.
Then $\varphi_h \in W_h$ and
\begin{equation*}
\begin{split}
  a(\psi_h,\varphi_h)&=-\left(\psi_h', ~\DDR1{\alpha-1}
  \left(-{_x\hspace{-0.3mm}I_1^{\alpha-1}}(P_W \psi_h')\right)\right)
  = (\psi_h',P_W \psi_h')\\
   &= (\psi_h', \psi_h') -  c_\alpha({_x\hspace{-0.3mm}I_1^{\alpha-1}}\psi_h',x^{1-\alpha}) (\psi_h',1)
   =\| \psi_h \|_U^2.
\end{split}
\end{equation*}
Further, the $L^2(D)$-stability of $P_W$ yields
\begin{equation*}
 \| \varphi_h \|_W = \| {_x\hspace{-0.3mm}I_1^{\alpha-1}}(P_W \psi_h') \|_W
  \le c\| P_W \psi_h' \|_{L^2(D)} \le c\| \psi_h \|_U.
\end{equation*}
Then we obtain \eqref{eqn:discr-inf-sup-triv-c} and the unique existence of a
solution $u_h\in U_h$.
\end{proof}

Next we introduce the (adjoint) Ritz projection $R_h: W\to W_h$ defined by
\begin{equation*}
a(\psi_h,R_h\varphi)=a(\psi_h,\varphi) \quad \forall \psi_h \in U_h.
\end{equation*}

Analogous to Lemma \ref{lem:Ritz}, the following error estimates hold for $R_h$.
\begin{lemma}\label{lem:Ritz-C}
The projection $R_h$ is well-defined and satisfies for any $\varphi\in W$
\begin{equation}\label{eqn:Ritz-prop-C}
  \begin{aligned}
    \|R_h \varphi\|_W & \le c\| \varphi\|_W,\\
    \| \varphi- R_h \varphi \|_{L^{2}(D)} & \le c h^{\alpha-1}   \| \varphi \|_W.
  \end{aligned}
\end{equation}
\end{lemma}
\begin{proof}
For any given $\varphi_h \in W_h$, let $ \psi_h= {_x\hspace{-0.3mm}I_1^{2-\alpha}} \varphi_h $.
Clearly, $\psi_h(1)=0$ and $\psi_h(0)=({_x\hspace{-0.3mm}I_1^{2-\alpha}} \varphi_h) (0) =
c_\alpha(x^{1-\alpha},\varphi_h)=0$ since $\varphi_h \in W_h \subset W$. Consequently
\begin{equation*}
  \| \varphi_h \|_{W} \le  c\|\DDR1{\alpha-1} \varphi_h\|_{L^2(D)} = c\| \psi_h \|_U.
\end{equation*}
The discrete inf-sup condition follows from this and \eqref{eqn:cap-inj}:
\begin{equation*}
   \sup_{\psi_h \in U_h}
  \frac{a(\psi_h,\varphi_h)}{ \|\psi_h\|_{U}}\geq c\|\varphi_h\|_W \quad \forall\varphi_h\in W_h.
\end{equation*}
The $\Hdi1{\alpha-1}$-stability of $R_h$ follows immediately
\begin{equation*}
      \|R_h \varphi\|_W \leq c\sup_{\psi_h \in U_h} \frac{a(\psi_h,R_h \varphi)}{ \|\psi_h\|_{U}}
         =c\sup_{\psi_h \in U_h} \frac{a(\psi_h,\varphi)}{ \|\psi_h\|_{U}} \le c\| \varphi \|_W.
\end{equation*}
Next let $g$ be the solution to \eqref{eqn:var-Caputo} with $F=\varphi-R_h\varphi$. By Theorem \ref{thm:regcap},
$\|g\|_{H^\alpha(D)}\leq c\|\varphi-R_h\varphi\|_{L^2(D)}$. Then by the Galerkin
orthogonality, Lemmas \ref{lem:Ritz-C} and \ref{fem-interp-U}, we deduce
\begin{equation*}
  \begin{aligned}
    \| \varphi- R_h \varphi \|_{L^{2}(D)}^2 & = a(g,\varphi- R_h \varphi) \le c \inf_{\psi_h\in U_h}\|  g- \psi_h \|_U \| \varphi-R_h\varphi \|_W\\
    &\le c h^{\alpha-1}   \| \varphi- R_h \varphi \|_{L^{2}(D)} \| \varphi \|_W.
  \end{aligned}
\end{equation*}
\end{proof}

\begin{remark}
The $L^2(D)$ error estimate of $R_h$ in the Caputo case is optimal, since
its adjoint problem has full regularity pickup.
\end{remark}
\begin{lemma}\label{lem:disinfsup:caputo}
Let Assumption \ref{ass:caputo} hold, $f\in L^2(D)$, and $b,q\in
L^\infty(D)$.   Then there exists an $h_0$ such that for all $h\le h_0$
\begin{equation*}%\label{discr-inf-sup-C}
  c \|\psi_h\|_U \le \sup_{\varphi_h \in W_h} \frac{A(\psi_h,\varphi_h)}{ \|\varphi_h\|_{W}} \quad \forall\psi_h\in U_h.
\end{equation*}
For such $h$, the finite element problem: Find $u_h\in U_h$ such that
\begin{equation}\label{discp-C}
A(u_h,\varphi_h)=(f,\varphi_h) \quad \forall\varphi_h\in W_h,
\end{equation}
has a unique solution.
\end{lemma}
\begin{proof}
The proof is the same as that of Lemma \ref{lem:disinfsup:riem}, using Lemmas
\ref{lem:inf-sup-disc-cap-triv} and \ref{lem:Ritz-C}.
\end{proof}

\begin{lemma}\label{lem:adjoint-error-C}
Let Assumption \ref{ass:caputo} hold, $f\in L^2(D)$, $b \in W^{1,\infty}(D)$ and
$q\in L^\infty(D)$. Let $w$ be the solution of  problem \eqref{eqn:var-adj-C}. Then there holds
\begin{equation*}
    \inf_{\psi_h \in W_h} \|w-\psi_h\|_{L^2(D)} + \inf_{\psi_h \in W_h}\|w-\psi_h\|_W %\|\DDR1 {\alpha-1} (w-\psi_h)\|_{L^2(D)}
    \le c h \|f\|_{L^2(D)}.
\end{equation*}
\end{lemma}
\begin{proof}
The proof is identical with that of Lemma \ref{lem:adjoint-error-RL}, with $P_W$ in place of $P_V$.
\end{proof}

\begin{theorem}\label{thm:femc}
Let $s\in[0,1/2)$ and Assumption \ref{ass:caputo} hold. Suppose $f\in H^s(D)$, $b \in W^{1,\infty}(D)$
and $q\in L^\infty(D)\cap H^s(D)$. Then there exists an $h_0$ such
that for all $h\le h_0$, the solution $u_h$ to the finite element problem \eqref{discp-C} satisfies
\begin{equation*}
   \|u-u_h\|_{L^2(D)} + h\|u-u_h\|_U  \le ch^{\min(\alpha+s,2)} \|f\|_{H^{s}(D)}.
\end{equation*}
\end{theorem}
\begin{proof}
The $H^1(D)$-estimate follows directly from Cea's lemma and Lemma \ref{lem:adjoint-error-C}
as in Theorem \ref{thm:femrl}. With $w$ being the solution to problem \eqref{eqn:var-adj-C}
with  $F=u-u_h$, by Theorem \ref{thm:regcap}, Lemmas \ref{fem-interp-U} and \ref{lem:adjoint-error-C}, we deduce
\begin{equation*}
\begin{split}
    \| u-u_h \|_{L^2(D)}^2 &= A( u-u_h,w)
    \le c\|u-u_h\|_U \inf_{\varphi_h \in W_h}\|w-\varphi_h\|_W \\%\| \DDR1 {\alpha-1}(w-\varphi_h) \|_{L^2(D)}\\
    & \le ch^{\min(\alpha+s,2)} \| f \|_{H^{s}(D)}\| u-u_h  \|_{L^2(D)}.
\end{split}
\end{equation*}
\end{proof}

%%%%%%%%%%%%%%%%%%%%%%%%%%%ssssssss
\subsection{Numerical implementation}
%%%%%%%%%%%%%%%%%%%%%%%%%%%ssssssss
Now we briefly discuss the efficient implementation of the Petrov-Galerkin FEM on a uniform mesh, especially the
computation of the stiffness matrix $\mathbf{S}=\mathbf{A}+\mathbf{R}:=[a_{i,j}]+[r_{i,j}]$, with
$$  a_{i,j} = -(\psi_j',~\DDR1{\alpha-1}  \varphi_i) ~~\text{and}~~ r_{i,j}=(b\psi_j', \varphi_i) + (q\psi_j, \varphi_i),$$
where $\psi_j$ and $\varphi_i$ are the basis functions in $U_h$ and in $V_h$ or $W_h$, respectively.
In the space $U_h$, we choose the nodal basis function $\{\psi_j\}$:
\begin{equation*}%\label{basis:Uh}
 \psi_j= \left\{ \begin{array}{rcl}
                   (x-x_{j-1})/h  &\quad \text{for}~~ x\in [x_{j-1},x_j), \\
                   (x_{j+1}-x)/h  &\quad \text{for}~~ x\in[x_{j},x_{j+1}), \\
                   0  & \quad \text{otherwise,}
                \end{array}\right.
\end{equation*}
with $j=1,2,..,m-1$.
Now, we set the basis function $\varphi_i$ of $V_h$ by
\begin{equation*}%\label{basis:Vh}
  \varphi_i= (x_{i}-x)^{\alpha-1}\chi_{[0,x_i]}-x_{i}^{\alpha-1}(1-x)^{\alpha-1}, \quad ~~i=1,2,..,m-1.
\end{equation*}
Clearly, $\varphi_i\in V_h$, and $\DDR1{\alpha-1} \varphi_i$ is piecewise constant
$-\DDR1{\alpha-1} \varphi_i = -\Gamma(\alpha) \chi_{[0,x_i]} + x_{i}^{\alpha-1}\Gamma(\alpha).$
Hence we have
\begin{equation}\label{eqn:matA}
  a_{i,j}= -(\psi_j',~\DDR1{\alpha-1} \varphi_i) = -\Gamma(\alpha)\delta_{ij}/h,
\end{equation}
where $\delta_{ij}$ is the Kronecker symbol. That is, the matrix $\mathbf{A}$ is a multiple
of the identity matrix, which is one distinct feature of the proposed approach. Hence, the
resulting linear system is well conditioned, and $\mathbf{A}$ can be used as a preconditioner,
if desired. The matrix $\mathbf{R}$ can be accurately computed using quadrature rules.

Likewise, we define the basis function $\varphi_i$ in $W_h$ by
\begin{equation*}%\label{basis:Vh}
  \varphi_i= (x_{i}-x)^{\alpha-1}\chi_{[0,x_i]}-x_{i} (1-x)^{\alpha-1} \quad \text{with}   ~~i=1,2,..,m-1.
\end{equation*}
By $(\varphi_i,x^{1-\alpha})=0$, $\varphi_i \in W_h$, and thus \eqref{eqn:matA} holds also in the Caputo case.

\section{An enriched FEM in the Riemann-Liouville case}\label{sec:sing}
By Theorem \ref{thm:femrl}, in the Riemann-Liouville case the FEM can only converge slowly,
due to the presence of the singular term $x^{\alpha-1}$. Now we discuss how to improve the
convergence, using an idea first introduced in \cite{CaiKim:2001} for the Poisson equation
on an L-shaped domain, and then extended to FBVPs in \cite{JinZhou:2014}. Below we only sketch
the technique and state the result, since the proofs are analogous to \cite{JinZhou:2014}.

This technique is to split the solution $u$ to problem \eqref{eqn:fde} into a regular
part $u^r$ and a singular part involving $x^{\alpha-1}$ (with $u^s=x^{\alpha-1}-x^2$):
\begin{equation*}
  u(x)=u^r + \mu u^s.
\end{equation*}
We shall assume ${_0\hspace{-0.3mm}I^\alpha_x} (b(u^s)'+qu^s) (1) \neq -1.$
Otherwise, we may replace the choice $x^2$ by any other function $v$ in the space $\Hdi0s$, $s\geq2$, with $v(1)=1$, such
that ${_0\hspace{-0.3mm}I^\alpha_x} (b(x^{\alpha-1}-v)'+q(x^{\alpha-1}-v)) (1) \neq -1$.
Then the regular part $u^r$ is given by
\begin{equation}\label{eqn:rep-ur}
u^r=-{_0\hspace{-0.3mm}I^\alpha_x} (f-bu'-qu)  + ({_0\hspace{-0.3mm}I^\alpha_x} (f-bu'-qu))(1) x^2,
\end{equation}
and the singularity strength $\mu$ is given by $\mu=c_0 \left({_0\hspace{-0.3mm}I^\alpha_x}
(f-b(u^r)'-qu^r)\right) (1),$ where $c_0= 1/(1+{_0\hspace{-0.3mm}
I^\alpha_x}(b(u^s)'+qu^s)(1)).$  The regular part $u^r$ satisfies
\begin{equation*}
  \begin{aligned}
    -\DDR 0 \alpha u^r + b(u^r)'+qu^r + \left({_0\hspace{-0.3mm}I^\alpha_x}(b(u^r)'+qu^r)\right) (1)Q  = \tilde {f} \quad \mbox{in } D,\\
   \end{aligned}
\end{equation*}
with  $u^r(0)=u^r(1)=0$, where the functions $Q$ and $\tilde{f}$ are defined by
$Q =c_0 {\DDR 0 \alpha} u^s -c_0b(u^s)'-c_0qu^s\in L^{2}(D)$ and
$\widetilde{f}=f+c_0\left({_0\hspace{-0.3mm}I^\alpha_x}f\right)(1)({\DDR 0 \alpha} u^s-b(u^s)'-qu^s)\in L^2(D)$,
respectively. Now we introduce a bilinear form $A_r(\cdot,\cdot):U\times V\to\mathbb{R}$ by
\begin{equation*}
    A_r(u,\varphi) = a(u,\varphi)+ b(u,\varphi),
\end{equation*}
with $  b(u,\varphi)= (bu'+qu,\varphi)+ {_0\hspace{-0.3mm}I^\alpha_x}(bu'+qu)(1)(Q,\varphi).$
Then the variational formulation of the regular part $u^r$ is to find $u^r \in U$ such that
\begin{equation}\label{eqn:var2}
 A_r(u^r,\varphi)= (\widetilde{f},\varphi)\quad \forall \varphi\in V.
\end{equation}

The following assumption on $A_r(\cdot,\cdot)$ is analogous to Assumption \ref{ass:riem}.
\begin{assumption} \label{ass:riem2}
Let the bilinear form $A_r(\cdot,\cdot): U\times V\to\mathbb{R}$ satisfy
\begin{itemize}
 \item[{$\mathrm{(a)}$}]  The problem of finding $u \in U$ such that $A_r(u,\varphi)=0$ for all $\varphi \in V$
           has only the trivial solution $u\equiv 0$.
 \item[{$(\mathrm{a}^\ast)$}] The problem of finding $\varphi \in V$ such that $A_r(u,\varphi)=0$ for all $u \in U$
    has only the trivial solution $\varphi \equiv 0$.
\end{itemize}
\end{assumption}

Under Assumption \ref{ass:riem2}, problem \eqref{eqn:var2} is stable and has extra regularity pickup.
\begin{theorem}\label{thm:regrl-new}
Let Assumption \ref{ass:riem2} hold, $b, q\in \Hd \gamma \cap L^\infty(D)$ and $f\in \Hdi0{\gamma}$
with $\gamma>\alpha-3/2$. Then there exists a unique solution $u^r\in H^{2\alpha-2+\beta}(D)\cap \Hd 1$,
$\beta\in[2-\alpha,1/2)$, to problem \eqref{eqn:var2} and further, it satisfies
\begin{equation*}
  \|u^r\|_{H^ {2\alpha-2+\beta}(D)} \le c\|f\|_{\Hdi 0\gamma}.
\end{equation*}
\end{theorem}
\begin{proof}
The proof of the stability is similar to that in Section \ref{subsec:var-RL}.
The regularity estimate follows from the representation \eqref{eqn:rep-ur} and Theorem \ref{thm:reg-RL}.
\end{proof}

Now we consider the discrete problem: find $u^r_h\in U_h$ such that
\begin{equation}\label{eqn:var2-d}
 A_r(u^r_h,\varphi_h)= (\widetilde{f},\varphi_h)\quad \forall \varphi_h\in V_h.
\end{equation}
Then we reconstruct $u_h$ by
\begin{equation}\label{eqn:muh}
   u_h = u^r_h + \mu_h u^s\quad \mbox{with}\quad \mu_h=c_0 \left({_0\hspace{-0.3mm}I^\alpha_x} (f-bu^r_h{'}-qu^r_h)\right) (1).
\end{equation}

Last, we state error estimates of the approximation $u_h$.
\begin{theorem}\label{thm:femrl-singrecon}
Let Assumption \ref{ass:riem2} hold, $b\in W^{1,\infty}(D)$, $q\in H^{1}(D)\cap
L^\infty(D)$ and $f\in \Hdi0{\gamma}$, $\gamma> \alpha-3/2$. Then there exists an $h_0$ such
that for all $h\le h_0$, the solution $u_h$ to problem \eqref{eqn:var2-d}-\eqref{eqn:muh}
satisfies for any $\beta\in(2-\alpha,1/2)$,
\begin{equation*}
   \|u-u_h\|_{L^2(D)} + h\|u-u_h\|_U  \le ch^{\min(2\alpha-2+\beta,2)} \|f\|_{\Hdi 0 \gamma}.
\end{equation*}
\end{theorem}

%%%%%%%%%%%%%%%%%%%%%%%%%%%ssssssss
\section{Numerical results and discussions}\label{sec:numer}
%%%%%%%%%%%%%%%%%%%%%%%%%%%ssssssss
Now we present numerical experiments to verify the convergence theory, and consider the following three examples:
\begin{itemize}
\item[(a)] The source term $f=x\in \Hdi0 {s}, \, s\in(1,3/2)$.
\item[(b)] The source term $f=1\in \Hdi0 {s}, \, s\in(0,1/2)$.
\item[(c)]  The source term $f=x^{-1/4}\in\Hdi0 {s},\, s\in(0,1/4)$.
\end{itemize}
The numerical results are computed on a uniform mesh with a mesh size $h=1/m$, $m\in \mathbb{N}$.
All the numerical experiments are performed on a personal computer with \texttt{MATLAB} 2014a.
In the case of $q,b\equiv0$, the exact solution is available in closed form, cf.
\eqref{eqn:solrep-RL} and \eqref{eqn:solrep-Cap}. In general, the analytic solution is not
available, and a reference solution is computed on a much finer mesh with a mesh size $h=1/5000$.

\subsection{Numerical results for example (a)}
The numerical results for case (a) with $b,q\equiv0$
are given in Tables \ref{tab:smooth-triv-RL} and \ref{tab:smooth-triv-C} for the Riemann-Liouville
and Caputo derivative, respectively. The notation \texttt{rate} in the tables refers to empirical
convergence rate, and the numbers in the bracket denote the theoretical predictions from Section
\ref{sec:fem}. The empirical rates agree well with the theoretical ones for all three
fractional orders. As the order $\alpha$ increases, the convergence rate in the $L^2(D)$
and $\Hd 1$-norm improves accordingly.
In the Riemann-Liouville case, despite the smoothness of the source term $f$, the solution
regularity is limited, due to the presence of the singularity $x^{\alpha-1}$. These observations
remain valid for the Caputo derivative, but the convergence rates are higher. The estimates
in Section \ref{sec:fem} are sharp for both derivatives. Further, we have the
following interesting observation: for $i=1,2,...,m-1$
\begin{equation*}
  \Gamma(\alpha) u(x_i) = (u',\,{\DDR1 {\alpha-1}}\psi_i) = -(f,\psi_i) = (u_h',\,{\DDR1 {\alpha-1}}\psi_i)=\Gamma(\alpha) u_h(x_i).
\end{equation*}
That is, the solution $u_h$ coincides with the $P1$ Lagrange interpolation of $u$.
This partly implies optimality of the convergence rates in Section \ref{sec:fem}.
The presence of a smooth $b$ and $q$ does not affect the convergence rates, cf.
Tables \ref{tab:smooth-RL} and \ref{tab:smooth-C}.

\begin{table}[hbt!]
\caption{Numerical results for example (a) with a Riemann-Liouville derivative and $b,q=0$, $\alpha=1.6, 1.75, 1.9$, $h=1/m$.}
\label{tab:smooth-triv-RL}
\vspace{-.3cm}
\begin{center}
     \begin{tabular}{|c|c|cccccc|c|}
     \hline
     $\alpha$  & $m$ &$10$ &$20$ &$40$ & $80$ & $160$ & $320$  &rate \\
     \hline
     $1.6$ & $L^2$ &3.10e-3 &1.39e-3 &6.42e-4 &2.99e-4 &1.39e-4 &6.47e-5 &$\approx$ 1.10 ($1.10$) \\

     & $H^1$   &1.67e-1 &1.50e-1 & 1.35e-1&  1.21e-1& 1.07e-1& 9.33e-2 &$\approx$ 0.17 ($0.10$)\\
    \hline
     $1.75$ & $L^2$    &1.25e-3 &4.62e-4 &1.84e-4 &7.55e-5 &3.15e-5 & 1.32e-5 & $\approx$ 1.27 ($1.25$)\\

     & $H^1$    & 5.03e-2 &3.89e-2 &3.14e-2 &2.57e-2 &2.10e-2 & 1.70e-2 &$\approx$ 0.29 ($0.25$)\\
     \hline
     $1.9$ & $L^2$    &6.40e-4 &1.72e-4 &4.92e-5 &1.53e-5 &5.14e-6 & 1.83e-6 & $\approx$ 1.53 ($1.40$)\\

     & $H^1$    & 2.08e-2 &1.15e-2 &6.81e-3 &4.38e-3 &3.01e-3 & 2.14e-3 &$\approx$ 0.50 ($0.40$)\\
     \hline
     \end{tabular}
\end{center}
\end{table}

\begin{table}[hbt!]
\caption{Numerical results for example (a) with a Caputo derivative and $b,q=0$, $\alpha=1.6, 1.75, 1.9$, $h=1/m$.}
\label{tab:smooth-triv-C}
\vspace{-.3cm}
\begin{center}
     \begin{tabular}{|c|c|cccccc|c|}
     \hline
     $\alpha$  & $m$ &$10$ &$20$ &$40$ & $80$ & $160$ & $320$  &rate \\
     \hline
     $1.6$ & $L^2$ &6.88e-4 &1.72e-4 &4.30e-5 &1.08e-5 &2.69e-6 &6.71e-7 &$\approx$ 2.00 ($2.00$) \\

     & $H^1$   &2.18e-2 &1.09e-2 &5.45e-3 &2.72e-3 &1.33e-3 & 6.47e-4 &$\approx$ 1.02 ($1.00$)\\
    \hline
     $1.75$ & $L^2$    &6.28e-4 &1.57e-4 &3.93e-5 &9.81e-6 &2.45e-6 & 6.12e-7 & $\approx$ 2.00 ($2.00$)\\

     & $H^1$    & 1.99e-2 &9.93e-3 &4.97e-3 &2.48e-3 &1.22e-3 & 5.90e-4 &$\approx$ 1.02 ($1.00$)\\
     \hline
     $1.9$ & $L^2$    &5.67e-4 &1.42e-4 &3.54e-5 &8.86e-6 &2.21e-6 & 5.53e-7 & $\approx$ 2.00 ($2.00$)\\

     & $H^1$    & 1.79e-2 &8.97e-3 &4.48e-3 &2.24e-3 &1.10e-3 & 5.33e-4 &$\approx$ 1.02 ($1.00$)\\
     \hline
     \end{tabular}
\end{center}
\end{table}

\begin{table}[hbt!]
\caption{Numerical results for example (a) with a Riemann-Liouville derivative and $b=e^x$, $q=x(1-x)$, $\alpha=1.6, 1.75, 1.9$, $h=1/m$.}
\label{tab:smooth-RL}
\vspace{-.3cm}
\begin{center}
     \begin{tabular}{|c|c|cccccc|c|}
     \hline
     $\alpha$  & $m$ &$10$ &$20$ &$40$ & $80$ & $160$ & $320$  &rate \\
     \hline
     $1.6$ & $L^2$ &2.67e-3 &9.41e-4 &3.89e-4 &1.74e-4 &8.01e-5 &3.69e-5 &$\approx$ 1.13 ($1.10$) \\

     & $H^1$   &1.22e-1 &9.14e-2 & 7.93e-2&  6.65e-2& 5.81e-2& 5.00e-2 &$\approx$ 0.20 ($0.10$)\\
    \hline
    $1.75$ & $L^2$    &1.23e-3 &3.69e-4 &1.28e-4 &4.92e-5 &2.00e-5 & 8.29e-6 & $\approx$ 1.29 ($1.25$)\\

     & $H^1$    & 5.25e-2 &3.18e-2 &2.18e-2 &1.65e-2 &1.31e-2 & 1.04e-2 &$\approx$ 0.33 ($0.25$)\\
     \hline
     $1.9$ & $L^2$    &7.49e-4 &1.92e-4 &5.05e-5 &1.40e-5 &4.20e-6 & 1.37e-6 & $\approx$ 1.66 ($1.40$)\\

     & $H^1$    & 3.02e-2 &1.55e-2 &8.10e-3 &4.44e-3 &2.61e-3 & 1.65e-3 &$\approx$ 0.70 ($0.40$)\\
     \hline
     \end{tabular}
\end{center}
\end{table}

\begin{table}[hbt!]
\caption{Numerical results for example (a) with a Caputo derivative and $b=e^x$, $q=x(1-x)$, $\alpha=1.6, 1.75, 1.9$, $h=1/m$.}
\label{tab:smooth-C}
\vspace{-.3cm}
\begin{center}
     \begin{tabular}{|c|c|cccccc|c|}
     \hline
     $\alpha$  & $m$ &$10$ &$20$ &$40$ & $80$ & $160$ & $320$  &rate \\
     \hline
     $1.6$ & $L^2$ & 1.91e-3 &4.92e-4 &1.25e-4 &3.18e-5 &8.03e-6 &2.01e-6 &$\approx$ 1.99 ($2.00$) \\

     & $H^1$   &7.12e-2 &3.59e-2 &1.80e-2 &9.00e-3 &4.50e-3 &2.12e-3 &$\approx$ 1.03 ($1.00$)\\
    \hline
     $1.75$ & $L^2$    &1.03e-3 &2.59e-4 &6.49e-5 &1.62e-5 &4.06e-6 & 1.01e-6 & $\approx$ 2.00 ($2.00$)\\

     & $H^1$    & 4.18e-2 &2.10e-2 &1.05e-2 &5.27e-3 &2.63e-3 & 1.24e-3 &$\approx$ 1.00 ($1.00$)\\
     \hline
     $1.9$ & $L^2$   &7.22e-4 &1.81e-4 &4.53e-5 &1.13e-5 &2.83e-6 &7.04e-7 & $\approx$ 2.00 ($2.00$)\\

     & $H^1$    & 2.88e-2 &1.45e-2 &7.25e-3 &3.62e-3 &1.81e-3 & 8.52e-4 &$\approx$ 1.02 ($1.00$)\\
     \hline
     \end{tabular}
\end{center}
\end{table}

One distinct feature of the proposed approach is that the stiffness matrix for the leading
term is diagonal,  and the resulting linear system is well conditioned. To illustrate this, we
give in Table \ref{tab:cond} the condition numbers of the stiffness matrix for
$\alpha=1.55$, $1.75$ and $1.95$. It is observed that for either derivative, it is fairly small
for the whole range of fractional orders, and  independent of the
mesh size $h$. These results fully confirm the observations in Section 4.4.

\begin{table}[hbt!]
\centering
\caption{The condition number of the linear system for $b(x)=e^x$, $q(x)=x(1-x)$, $\alpha=1.55,\ 1.75$ and $1.95$, $h=1/m$. }\label{tab:cond}
\vspace{-.2cm}
\begin{tabular}{|c|c|ccccccc|}
\hline
 Deriv. type& $\alpha\backslash m$ & 20& 40  & 80  & 160 & 320& 640&1280\\
\hline
         &1.55 &  2.98  & 3.48 & 4.26 & 4.30 & 4.57 &4.84&5.00 \\
R-L        &1.75 &  2.06  & 2.22 & 2.33 & 2.40 & 2.45 &2.48&2.50\\
        & 1.95& 1.63 & 1.68 & 1.71 & 1.73& 1.74 &1.74&1.75\\
\hline
         &1.55 & 2.75& 3.20  & 3.57 & 3.89 & 4.16& 4.39&4.60\\
Caputo   &1.75 & 2.02& 2.17  & 2.27 & 2.34 & 2.39& 2.42&2.44\\
         & 1.95& 1.63 & 1.68 & 1.71 & 1.73& 1.73 &1.74&1.74\\
\hline
\end{tabular}
\end{table}

\subsection{Numerical results for example (b)}
Here the source term $f$ is smooth but does not satisfy the zero
boundary condition. In the Riemann-Liouville case, the $L^2(D)$ and $H^1(D)$ errors
are respectively of order $O(h^{\alpha-1/2})$ and $O(h^{\alpha-3/2})$, while in the Caputo
case, an $O(h)$ and $O(h^2)$ rate is observed for $L^2(D)$ and $H^1(D)$ errors,
respectively, cf. Tables \ref{tab:interm-RL} and \ref{tab:interm-C}, which fully confirm our convergence theory.

\begin{table}[hbt!]
\caption{Numerical results for example (b) with a Riemann-Liouville derivative and $b=e^x$, $q=x(1-x)$, $\alpha=1.6, 1.75, 1.9$, $h=1/m$.}
\label{tab:interm-RL}
\vspace{-.3cm}
\begin{center}
     \begin{tabular}{|c|c|cccccc|c|}
     \hline
     $\alpha$  & $m$ &$10$ &$20$ &$40$ & $80$ & $160$ & $320$  &rate \\
     \hline
     $1.6$ & $L^2$ &6.87e-3 &2.89e-3 &1.28e-3 &5.78e-4 &2.65e-4 &1.22e-4 &$\approx$ 1.12 ($1.10$) \\

     & $H^1$   &3.28e-1 &2.79e-1 & 2.46e-1&  2.17e-1& 1.90e-1& 1.63e-1 &$\approx$ 0.20 ($0.10$)\\
    \hline
    $1.75$ & $L^2$    &2.93e-3 &1.05e-3 &4.04e-4 &1.62e-4 &6.65e-5 & 2.76e-5 & $\approx$ 1.28 ($1.25$)\\

     & $H^1$    & 1.17e-1 &8.47e-2 &6.59e-2 &5.29e-2 &4.28e-2 & 3.42e-2 &$\approx$ 0.31 ($0.25$)\\
     \hline
     $1.9$ & $L^2$   &1.37e-3 &3.82e-4 &1.12e-4 &3.54e-5 &1.19e-5 &4.23e-6 & $\approx$ 1.53 ($1.40$)\\

     & $H^1$   &5.16e-2 & 2.84e-2 &1.66e-2 &1.04e-2 &7.02e-3 &4.91e-3  &$\approx$ 0.54 ($0.40$)\\
     \hline
     \end{tabular}
\end{center}
\end{table}

\begin{table}[hbt!]
\caption{Numerical results for example (b) with a Caputo derivative and $b=e^x$, $q=x(1-x)$, $\alpha=1.6, 1.75, 1.9$, $h=1/m$.}
\label{tab:interm-C}
\vspace{-.3cm}
\begin{center}
     \begin{tabular}{|c|c|cccccc|c|}
     \hline
     $\alpha$  & $m$ &$10$ &$20$ &$40$ & $80$ & $160$ & $320$  &rate \\
     \hline
     $1.6$ & $L^2$ & 1.86e-3 &4.72e-4 &1.19e-4 &3.00e-5 &7.54e-6 &1.88e-6 &$\approx$ 2.00 ($2.00$) \\

     & $H^1$   &7.89e-2 &3.98e-2 &2.00e-2 &1.00e-2 &5.00e-3 &2.35e-3 &$\approx$ 1.03 ($1.00$)\\
    \hline
     $1.75$ & $L^2$    &1.28e-3 &3.22e-4 &8.03e-5 &2.01e-5 &5.01e-6 & 1.25e-6 & $\approx$ 2.00 ($2.00$)\\

     & $H^1$    & 5.38e-2 &2.71e-2 &1.35e-2 &6.78e-3 &3.39e-3 & 1.59e-3 &$\approx$ 1.03 ($1.00$)\\
     \hline
     $1.9$ & $L^2$   &1.07e-3 &2.68e-4 &6.71e-5 &1.68e-5 &4.19e-6 &1.04e-6 & $\approx$ 2.00 ($2.00$)\\

     & $H^1$    & 4.20e-2 &2.11e-2 &1.05e-2 &5.28e-3 &2.64e-3 & 1.24e-3 &$\approx$ 1.04 ($1.00$)\\
     \hline
     \end{tabular}
\end{center}
\end{table}

\subsection{Numerical results for example (c)}
Note that the source term $f(x)=x^{-1/4} \in \Hd {s}$ with $s\in[0,1/4)$. Hence, in the Caputo
case with $\alpha<1.75$, the solution $u$ fails to be in $H^2(D)$, which deteriorates the convergence rate. The $H^1(D)$
and $L^2(D)$-errors are of order $O(h^{0.85})$ and $O(h^{1.85})$ in case of $\alpha=1.6$, while for
$\alpha=1.75$ and $1.9$, an $O(h)$ and $O(h^2)$ rate of the $H^1(D)$ and $L^2(D)$-errors
is observed, cf. Table \ref{tab:nonsmooth-C}, confirming theoretical predictions. In the Riemann-Liouville case, the
desired optimal but slow convergence behavior is observed, cf. Table \ref{tab:nonsmooth-RL}.

\begin{table}[hbt!]
\caption{Numerical results for example (c) with the Riemann-Liouville derivative and $b=e^x$, $q=x(1-x)$, $\alpha=1.6, 1.75, 1.9$, $h=1/m$.}
\label{tab:nonsmooth-RL}
\vspace{-.3cm}
\begin{center}
     \begin{tabular}{|c|c|cccccc|c|}
     \hline
     $\alpha$  & $m$ &$10$ &$20$ &$40$ & $80$ & $160$ & $320$  &rate \\
     \hline
     $1.6$ & $L^2$ &1.07-2 &4.61e-3 &2.04e-3 &9.18e-4 &4.18e-4 &1.91e-4 &$\approx$ 1.13 ($1.10$) \\

     & $H^1$   &5.03e-1 &4.38e-1 & 3.87e-1&  3.42e-1& 2.99e-1& 2.57e-1 &$\approx$ 0.20 ($0.10$)\\
    \hline
     $1.75$ & $L^2$    &4.62e-3 &1.71e-3 &6.60e-4 &2.63e-4 &1.07e-4 & 4.40e-5 & $\approx$ 1.30 ($1.25$)\\

     & $H^1$    & 1.76e-1 &1.33e-1  &1.05e-1 &8.47e-2 &6.83e-2 &5.44e-2 &$\approx$ 0.32 ($0.25$)\\
     \hline
     $1.9$ & $L^2$   &2.02e-3 &5.93e-4 &1.82e-4 &5.87e-5 &1.99e-5 &6.99e-6 & $\approx$ 1.53 ($1.40$)\\

     & $H^1$   &7.05e-2 & 4.12e-2 &2.54e-2 &1.66e-2 &1.14e-2 &7.97e-3  &$\approx$ 0.52 ($0.40$)\\
     \hline
     \end{tabular}
\end{center}
\end{table}

\begin{table}[hbt!]
\caption{Numerical results for example (c) with a Caputo derivative and $b=e^x$, $q=x(1-x)$, $\alpha=1.6, 1.75, 1.9$, $h=1/m$.}
\label{tab:nonsmooth-C}
\vspace{-.3cm}
\begin{center}
     \begin{tabular}{|c|c|cccccc|c|}
     \hline
     $\alpha$  & $m$ &$10$ &$20$ &$40$ & $80$ & $160$ & $320$  &rate \\
     \hline
     $1.6$& $L^2$ & 1.84e-3 &4.92e-4 &1.31e-4 &3.51e-5 &9.46e-6 &2.56e-6 &$\approx$ 1.89 ($1.85$) \\

     & $H^1$   &7.47e-2 &3.87e-2 &2.01e-2 &1.05e-2 &5.54e-3 &2.82e-3 &$\approx$ 0.94 ($0.85$)\\
    \hline
     $1.75$ & $L^2$    &1.56e-3 &4.05e-4 &1.05e-4 &2.72e-5 &7.04e-6 & 1.81e-6 & $\approx$ 1.97 ($2.00$)\\

     & $H^1$    & 5.92e-2 &3.04e-2 &1.56e-2 &7.99e-3 &4.09e-3 & 1.98e-3 &$\approx$ 0.99 ($1.00$)\\
     \hline
     $1.9$ & $L^2$   &1.39e-3 &3.54e-4 &8.99e-5 &2.28e-5 &5.74e-6 &1.44e-6 & $\approx$ 1.99 ($2.00$)\\

     & $H^1$    & 5.02e-2 &2.55e-2 &1.29e-2 &6.49e-3 &3.27e-3 & 1.55e-3 &$\approx$ 1.03 ($1.00$)\\
     \hline
     \end{tabular}
\end{center}
\end{table}

\subsection{Numerical results for the enriched FEM}
\begin{table}[hbt!]
\caption{Numerical results for $u^r$ for example (b) with a Riemann-Liouville derivative,
by the enriched FEM and $b=1$, $q=x(1-x)$, $\alpha=1.6, 1.75, 1.9$, $h=1/m$.}
\label{tab:sing-recon-b-ur}
\vspace{-.3cm}
\begin{center}
     \begin{tabular}{|c|c|cccccc|c|}
     \hline
     $\alpha$  & $m$ &$10$ &$20$ &$40$ & $80$ & $160$ & $320$  &rate \\
     \hline
     $1.6$ & $L^2$ &5.51e-4 &1.40e-4 &3.66e-5 &1.01e-5 &9.18e-6 &2.87e-6 &$\approx$ 1.74 ($1.70$) \\

     & $H^1$   &2.04e-2 &1.05e-2 & 5.50e-3 &  2.99e-3& 1.70e-3& 1.01e-3 &$\approx$ 0.78 ($0.70$)\\
     \hline
     $1.75$ & $L^2$    &3.86e-4 &9.84e-5 &2.52e-5 &6.46e-6 &1.66e-6 & 4.30e-7 & $\approx$ 1.96 ($2.00$)\\

     & $H^1$    & 1.38e-2 &6.98e-3 &3.55e-3 &1.79e-3 &9.04e-4 & 4.51e-4 &$\approx$ 1.02 ($1.00$)\\
     \hline
     $1.9$ & $L^2$   &3.00e-4 &7.52e-5 &1.88e-5 &4.71e-6 &1.18e-6 &2.94e-7 & $\approx$ 2.00 ($2.00$)\\

     & $H^1$   &1.04e-2 &5.21e-3 &2.61e-3 &1.29e-3 &6.33e-4 & 3.06e-4  &$\approx$ 1.03 ($1.00$)\\
     \hline
     \end{tabular}
\end{center}
\end{table}

\begin{table}[hbt!]
\caption{$|\mu-\mu_h|$ for example (b) with a Riemann-Liouville derivative, by the enriched
FEM, and $b=1$, $q=x(1-x)$, $\alpha=1.6, 1.75, 1.9$, $h=1/m$.}
\label{tab:sing-recon-b-la}
\vspace{-.3cm}
\begin{center}
     \begin{tabular}{|c|cccccc|c|}
     \hline
     $\alpha$  &$10$ &$20$ &$40$ & $80$ & $160$ & $320$  &rate \\
     \hline
     $1.6$  &4.04e-4 &1.00e-4 &2.44e-5 &5.88e-6 &1.41e-6 &3.37e-7 &$\approx$ 2.07 ($1.70$) \\
     \hline
     $1.75$   &1.59e-4 &3.90e-5 &9.54e-6 &2.34e-6 &5.74e-7 & 1.41e-7 & $\approx$ 2.02 ($2.00$)\\
     \hline
     $1.9$&8.58e-5 &2.14e-5 &5.36e-6 &1.34e-6 &3.34e-7 & 8.33e-8 & $\approx$ 2.00 ($2.00$)\\
     \hline
     \end{tabular}
\end{center}
\end{table}

In Table \ref{tab:sing-recon-b-ur}, we present the $L^2(D)$ and $H^1(D)$ norms of the error in
approximating the regular part $u^r$ for example (b).
Since $f\in \Hd s$ with $s\in[0,1/2)$, by Theorem \ref{thm:regrl-new}, $u^r \in H^2(D)$ in
case of $\alpha > 1.75$. The numerical results show a convergence rate of $O(h^2)$ and $O(h)$ for the $L^2(D)$ and
$H^1(D)$-norms of the error, respectively, for $\alpha=1.75$ and $1.9$. For $\alpha=1.6$, the
regular part $u^r$ lies in $H^{1.2+\beta}(D)$ with $\beta\in[2-\alpha,1/2)$ by Theorem \ref{thm:regrl-new},
and we observe a convergence rate $O(h^{1.7})$ and $O(h^{0.7})$, respectively, in $L^2(D)$
and $H^1(D)$-norm, which fully confirms Theorem \ref{thm:femrl-singrecon}. The error $|\mu-\mu_h|$
of the reconstructed singular strength $\mu_h$ achieves an $O(h^2)$ convergence, even for $\alpha
=1.6$,
cf. Table  \ref{tab:sing-recon-b-la}.

\section{Conclusions}
In this work, we have developed novel variational formulations for fractional BVPs involving
a convection term. The fractional derivative in the leading term is of either Riemann-Liouville
or Caputo type. The well-posedness and sharp regularity pickup of the formulations are established.
A new finite element method, using continuous piecewise linear finite elements and ``shifted''
fractional powers for the trial and test space, respectively, was also developed. It leads to a
diagonal stiffness matrix for the leading term (on a uniform mesh), and admits optimal $L^2(D)$
and $H^1(D)$ error estimates, which is the first FEM with such desirable properties. Further,
an enriched FEM was proposed to improve the convergence in the Riemann-Liouville case, and
optimal error estimates were provided. Extensive numerical experiments fully confirm the
convergence analysis.

There are several avenues for further research. First, it is of immense interest to
extend the approach to higher dimensions. This extension is formally feasible for the
Riemann-Liouville case. However, their solution theory, e.g., well-posedness and sharp
regularity pickup, is missing. Second, it is also of much interest to extend the approach
to other type or inhomogeneous boundary conditions, which may induce much graver solution
singularity. Third, the adaptation of the approach to the time-dependent problems,
including the space-time fractional model, is important. Especially, it may allow one
to derive optimal $L^2(D)$ error estimates.

\section*{Acknowledgements}
The research of B. Jin has been partly supported by NSF Grant DMS-1319052 and UK EPSRC
EP/M025160/1, and R. Lazarov and Z. Zhou was supported in parts by NSF Grant DMS-1016525.

\bibliographystyle{abbrv}
\bibliography{fracrev}
\end{document}